\documentclass[a4paper,12pt]{article}

\usepackage{graphicx}%
\usepackage{multirow}%
\usepackage{amsmath,amssymb,amsfonts}%
\usepackage{amsthm}%
\usepackage{mathrsfs}%
\usepackage[title]{appendix}%
\usepackage{xcolor}%
\usepackage{textcomp}%
\usepackage{manyfoot}%
\usepackage{booktabs}%
\usepackage{algorithm}%
\usepackage{algorithmicx}%
\usepackage{algpseudocode}%
\usepackage{listings}%

\theoremstyle{thmstyleone}%
\newtheorem{theorem}{Theorem}
\theoremstyle{thmstyletwo}%

\theoremstyle{thmstylethree}%
\newtheorem{definition}{Definition}%

\newtheorem{lemma}{Lemma} 

\newcommand{\bb}{\mbox{\boldmath$b$}}
\newcommand{\be}{\mbox{\boldmath$e$}}

\newcommand{\br}{\mbox{\boldmath$r$}}

\newcommand{\bv}{\mbox{\boldmath$v$}}
\newcommand{\bw}{\mbox{\boldmath$w$}}
\newcommand{\bx}{\mbox{\boldmath$x$}}
\newcommand{\by}{\mbox{\boldmath$y$}}
\newcommand{\bz}{\mbox{\boldmath$z$}}

\newcommand{\nul}{{\cal N}}
\newcommand{\ran}{{\cal R}}
\newcommand{\rank}{\mbox{rank}}

\newcommand{\real}{{\bf R}}

\newcommand{\rk}{{\bf R}^k}
\newcommand{\rmn}{{\bf R}^{m \times n}}
\newcommand{\Rm}{{\bf R}^m}
\newcommand{\rn}{{\bf R}^n}
\newcommand{\rr}{{\bf R}^r}
\newcommand{\rnm}{{\bf R}^{n \times m}}
\newcommand{\rnn}{{\bf R}^{n \times n}}
\newcommand{\rnr}{{\bf R}^{n \times r}}
\newcommand{\rrr}{{\bf R}^{r \times r}}

\newcommand{\rkk}{{\bf R}^{k \times k}}

\newcommand{\argmin}{\mathop{\rm arg~min}\limits}
\newcommand{\rnnr}{{\bf R}^{n \times (n-r)}}

\def\vector#1{\mbox{\boldmath $#1$}}

\begin{document}

\title{NR-SSOR right preconditioned RRGMRES for non-range-symmetric systems 
and rank-deficient rectangular least squares problems}


\author{Kota Sugihara\thanks{No Affiliation, E-mail: kouta.sugihara@gmail.com} \,
and Ken Hayami\thanks{Professor Emeritus, National Institute of Informatics, and 
The Graduate University for Advanced Studies (SOKENDAI),
Tokyo, Japan, E-mail: hayami@nii.ac.jp}
}










\maketitle

\abstract{GMRES is known to determine a least squares (LS) solution of $ A \bx = \bb $ 
where  $ A \in \rnn $ without breakdown for arbitrary $ \bb \in \rn$, and 
initial iterate $\bx_0 \in \rn $ 
if and only if $A$ is range-symmetric, i.e.
$ \ran(A^{\rm T}) = \ran(A) $ holds, where $ A $ may be singular 
and $\bb $ may not be in the range space $\ran(A)$ of $A$. 
On the other hand, when $ \ran(A^{\rm T}) \neq \ran(A) $,
including the case when the index of $A$ is greater than or equal to 1,
there exist $\bx_0$ and $\bb$ such that GMRES breaks down without giving 
a LS solution.

In this paper, to solve this difficulty, we propose applying the Range 
Restricted GMRES (RRGMRES) to $ A C A^{\rm T}  \bz = \bb $, where 
$ C \in \rnn $ is symmetric positive definite (SPD).
The proposed RRGMRES robustly determines 
a least squares solution $ \bx = CA^{\rm T} \bz$ of $ A \bx = \bb $ 
without breakdown for 
arbitrary (singular) matrix $A \in \rnn$ and $\bb, \bx_0 \in \rn$ 
since $ A C A^{\rm T}$ is symmetric even if $ \ran(A^{\rm T}) \neq \ran(A) $.
In particular, we prove that $C \in \rnn$ for NR-SSOR right preconditioner 
is SPD
and propose the NR-SSOR right preconditioned RRGMRES,
which also works efficiently for rectangular (rank-deficient) least squares problems 
$ \min_{\bx \in \rn} \| \bb - A \bx \|_2 $ for $ A \in \rmn$ and arbitrary $\bb \in \Rm$.

Numerical experiments demonstrate the validity of the proposed method.}

{\bf Key words.}
GMRES; Singular systems; Range-symmetry; GP; Index; 
 RRGMRES; Right preconditioning; NR-SSOR; Least squares problems; MINRES-QLP  





\section{Introduction}\label{sec1}

The Generalized Minimal Residual (GMRES) method \cite{SS} is a robust and efficient Krylov subspace 
iterative method for systems of linear equations
\begin{equation}\label{lineq}
A \bx = \bb
\end{equation}
 where $ A \in \rnn$ is nonsingular and may be nonsymmetric, and $\bx, \bb \in \rn$. 

Abstractly, GMRES begins with an initial approximate solution $ \bx_0 \in \rn$ and initial residual $ \br_0 = \bb - A \bx_0 $ and 
characterizes the $k$th approximate solution as $\bx_k = \bx_0 + \bz_k$, where $\bz_k$ solves
\begin{equation}
 \min_{ \bz \in {\cal K}_k } \| \bb - A ( \bx_0 + \bz ) \|_2 
  = \min_{ \bz \in {\cal K}_k } \| \br_0 - A \bz \|_2 .
 \label{GMRES}
\end{equation}
Here, ${\cal K}_k$ is the $k$th Krylov subspace determined by $A$ and $\br_0$, defined by
\[ {\cal K}_k \equiv {\rm span} \{ \br_0, A\br_0, \ldots, A^{k-1} \br_0 \}. \]

The implementation given in \cite{SS} is as follows.
\\ \\
{\bf Algorithm 1: GMRES}\\ \\
Choose $\bx_0$. 
$ \br_0 = \bb - A \bx_0 $, 
$ \bv_1 = \br_0 / ||\br_0||_2 $ \\
For $ j = 1, 2, \cdots $ until satisfied do \\
\hspace{5mm} $ h_{i,j}=(\bv_i,A \bv_j )\hspace{4mm}(i=1,2,\ldots,j) $ \\ 
\hspace{5mm}
$ {\displaystyle \hat{\bv}_{j+1} = A \bv_j - \sum_{i=1}^j h_{i,j} \bv_i } $ \\
\hspace{5mm} $ h_{j+1,j} =||\hat{\bv}_{j+1} ||_2 $.
\hspace{4mm}If $ h_{j+1,j} =0, $ goto $ \ast $. \\
\hspace{5mm} $ \bv_{j+1} = \hat{\bv}_{j+1} / h_{j+1,j} $ \\
End do \\
$ \ast \, k:=j $ \\
Form the approximate solution \\ 
\hspace{5mm} $ \bx_k = \bx_0 + [ \bv_1,\ldots,\bv_k] \by_k $ \\
where $ \by = \by_k $ minimizes $  ||\br_k||_2 = ||\beta\be_1-\overline{H}_k \by ||_2 . $
\begin{equation}
\nonumber
\end{equation}

Here, $\vector{r}_{k} = \vector{b} - A\vector{x}_{k}$, and $ \overline{H}_k = [ h_{i,j} ] \in {\bf R}^{ (k+1) \times k}$ is a 
Hessenberg matrix, i.e., $h_{i,j}=0$ for $i>j+1$. \, \,
$\beta={||\br_0 ||_2} $ \, and \, $ \be_1 = [1,0,\ldots,0]^{\rm T} \in {\bf R}^{k+1} $.

When $A$ is nonsingular, GMRES gives the solution for arbitrary $\vector{b}$ and $\vector{x}_0$.
When $ A $ is singular, one may still apply GMRES to the least squares problem
\begin{equation}
 \min_{\bx \in \rn} \| \bb - A \bx \|_2 .
 \label{lstsq}
\end{equation} 
However, Brown and Walker \cite[Theorem 2.4]{BW} showed that GMRES determines a least squares solution of (\ref{lstsq}) without 
breakdown for all $ \bb $ and initial approximate solution $ \bx_0 $ if and only if $ A $ is range-symmetric (EP), 
that is $ \ran(A)=\ran(A^{\rm T})$, where $ \ran(A) $ denotes the range space of $ A $. (See also \cite{HS} for an 
alternative proof.)

There are two kinds of definitions for the breakdown of GMRES.
One definition is that
GMRES breaks down at the $k$th iteration when 
${\rm dim}A({\cal K}_k) < {\rm dim}{\cal K}_k$ or ${\rm dim}{\cal K}_k < k$ \cite{BW}.
The other definition is that GMRES breaks down at the $k$th iteration when $ h_{k+1,k} = 0$ in Algorithm 1.
The equivalence between these definitions for the breakdown of GMRES is discussed in \cite{MH}. 
In this paper, we will use the latter definition.

When $ \ran(A^{\rm T}) \neq \ran(A) $,
including the case when the index of $A$ is greater than or equal to 1,
there exist $\bx_0$ and $\bb$ such that GMRES breaks down without giving 
a least squares solution \cite{HS}.
In order to overcome this problem, Reichel and Ye \cite{RY} proposed the 
breakdown-free GMRES, which is designed to expand 
the solution Krylov subspace when GMRES breaks down. However, they do not give 
a theoretical justification that 
this method always works.

In this paper, we consider
solving square singular systems where
$ \ran(A^{\rm T}) \neq \ran(A) $, especially where $A$ is GP or
the index of $A$ is greater than $1$.
To do so, 
we propose the method to overcome the difficulty that
GMRES does not necessarily determine a least squares solution when 
$ \ran(A^{\rm T}) \neq \ran(A)$ 
and give a theoretical justification that 
our proposed method always works.
The proposed method is as follows.

For the least squares problem 
\begin{equation}
 \min_{\bx \in \rn} \| \bb - A \bx \|_2
 \label{lstsq2}
\end{equation} 
where $ A \in \rmn $, Hayami, Yin and Ito \cite{HYI} proposed the AB-GMRES method which applies GMRES to 
\begin{equation}
 \min_{\bz \in \Rm} \| \bb -AB \bz \|_2 
 \label{AB}
\end{equation}
where $ B \in \rnm $ satisfies $ \ran (AB) = \ran (A) $.

If we let $ m=n $, and $ B=CA^{\rm T} $, where $ C \in \rnn $ is symmetric positive-definite, then $ \ran(A) = \ran (AB) $
 holds (\cite[Lemma 3.4]{HYI}).
Note also that $ \ran ( (AB)^{\rm T} )= \ran (AB)$, i.e. $ AB $ becomes range-symmetric, so that the AB-GMRES determines a least 
squares solution of (\ref{AB}) for arbitrary $ \bb \in \Rm$. Note also that $B=CA^{\rm T}$ can be regarded as a right preconditioner.
Therefore, instead of $A\vector{x} = \vector{b}$, we apply GMRES to
$ACA^{\rm T}$ where $C$ is symmetric positive definite. 
Then, AB-GMRES determines a least squares solution 
$ \bx = CA^{\rm T} \bz$ of $ A \bx = \bb $ 
without breakdown for 
arbitrary (singular) matrix $A \in \rnn$ and $\bb, \bx_0 \in \rn$ 
since $ A C A^{\rm T}$ is symmetric even if $ \ran(A^{\rm T}) \neq \ran(A) $.

Note however, that 
even when $\ran ( A^{\rm T} )= \ran (A) $, when $ \bb \notin \ran (A)$, the least squares problem (\ref{GMRES}) 
in GMRES may 
become dangerously ill-conditioned before a least squares solution of (\ref{lstsq}) is obtained (\cite[Theorem 2.5]{BW}).

To overcome this difficulty, we further propose using the Range Restricted GMRES (RRGMRES) \cite{CLR} instead of GMRES 
because
RRGMRES is much more stable compared to GMRES for inconsistent range-symmetric systems.

The $k$th iterate of RRGMRES is given by $\vector{x}_k = \vector{x}_0 + \vector{z}_k$ where
$\vector{z}_k$ solves \\ $\displaystyle \min_{\vector{z} \in {\cal K}_k(A, A\vector{r}_0)}\|\vector{r}_0 - A\vector{z}\|_2$.
The implementation of RRGMRES given in \cite{CLR} as follows.
\\ \\
{\bf Algorithm 2: RRGMRES}\\ \\
Choose $\bx_0$. $ \br_0 = \bb - A \bx_0 $, \\
$ \bv_1 = A\br_0 / ||A\br_0||_2 $ \\
For $ j = 1, 2, \cdots $ until satisfied do \\
\hspace{5mm} $ h_{i,j}=(\bv_i,A \bv_j )\hspace{4mm}(i=1,2,\ldots,j) $ \\ 
\hspace{5mm}
$ {\displaystyle \hat{\bv}_{j+1} = A \bv_j - \sum_{i=1}^j h_{i,j} \bv_i } $ \\
\hspace{5mm} $ h_{j+1,j} =||\hat{\bv}_{j+1} ||_2 $.
\hspace{4mm}If $ h_{j+1,j} =0, $ goto $ \ast $. \\
\hspace{5mm} $ \bv_{j+1} = \hat{\bv}_{j+1} / h_{j+1,j} $ \\
End do \\
$ \ast \, k:=j $ \\
Form the approximate solution \\ 
\hspace{5mm} $ \bx_k = \bx_0 +  [ \bv_1,\ldots,\bv_k] \by_k $ \\
where $ \by = \by_k $ minimizes $  ||\br_k||_2^2 = \|H_{k+1,k}\vector{y} - V_{k+1}^{\rm T}\vector{r}_0\|_{2}^{2} 
  + \|(I_{n} - V_{k+1}V_{k+1}^{\rm T})\vector{r}_0\|_2^{2}. $
\begin{equation}
\nonumber
\end{equation}

Here, $V_{k+1} =  [ \bv_1,\ldots,\bv_{k+1}] \in {\bf R}^{n \times (k+1)}$, and $\vector{r}_{k} = \vector{b} - A\vector{x}_{k}$, 
and $ H_{k+1,k} = [ h_{i,j} ] \in {\bf R}^{ (k+1) \times k}$ is a 
Hessenberg matrix, i.e., $h_{i,j}=0$ for $i>j+1$, and 
$I_n \in \rnn$ is an identity matrix. 
Furthermore, the last line of {\bf Algorithm 2} is justified by Lemma \ref{RReq} in Appendix \ref{rrnonsg}.

The AB-GMRES with $B = A^{\rm T}C$ using NE-SOR as the inner-iteration preconditioner determines a solution of 
$A\vector{x} = \vector{b}$ without breakdown for all $\vector{b} \in \ran(A)$
and all initial iterate $\vector{x}_0 \in \rn$ \cite{MH}. 
However, the same 
method does not necessarily determine 
a least squares solution of $A\vector{x} = \vector{b}$ when 
$\vector{b} \notin \ran(A)$ since $\ran(AB) = \ran((AB)^{\rm T})$ does 
not necessarily hold when $C$ is not an identity matrix.

To overcome this difficulty of AB-GMRES with $B = A^{\rm T}C$ 
using NE-SOR for inconsistent systems, and
speed up the convergence of AB-GMRES and AB-RRGMRES,
 we propose using $B = CA^{\rm T}$ with NR-SSOR as the inner iteration right preconditioner, 
and show that the corresponding $C$ matrix is symmetric positive definite
when $A$ has no zero columns and $0 < \omega < 2$ where $\omega$ is the relaxation parameter
of NR-SSOR.
Therefore, we can prove that NR-SSOR right preconditioned RRGMRES determines 
a least squares solution of (\ref{lstsq2}) for arbitrary $ \bb \in \rn$,
 including the inconsistent case
even if $\ran(A) \neq \ran(A^{\rm T})$.
Furthermore, 
the proposed method also works well for rectangular least squares problems 
$ \min_{\bx \in \rn} \| \bb - A \bx \|_2 $ where $ A \in \rmn$ and $\vector{b} \in \Rm$.
Numerical experiments with systems with non-range-symmetric $A$ with ${\rm index}(A) \geq 1$
and underdetermined ($m < n$) least squares problems,
show the validity of the proposed method.

The rest of the paper is organized as follows. In section \ref{GpEp}, we explain the definition of index, GP, EP matrices,
underdetermined problems and
their applications.
In section \ref{ConvGm}, we review the theory for GMRES on singular systems. 
In section \ref{RpreGrr}, we propose the right preconditioned RRGMRES for arbitrary singular systems, including the case when the index of $A$ is greater than
or equal to $1$. In section \ref{NRssApp}, we propose 
applying the NR-SSOR as the inner-iteration right preconditioner
where $B = CA^{\rm T}$,
and show that the corresponding $C$ matrix is symmetric positive definite.
In section \ref{MINQLP}, we explain the MINRES-QLP method.
We present numerical experiment results for the proposed method for GP systems in section \ref{NumGP},
and for index 2 systems in section \ref{NumId2}, and for underdetermined least squares problems in section \ref{NumLS}
and compare with the MINRES-QLP method \cite{CPS}.
In section \ref{ConcL}, we conclude the paper. In the Appendix \ref{rrnonsg} and \ref{rrsg}, 
we give the convergence theory for RRGMRES
for nonsingular and singular systems.

\section{GP, EP matrices and their applications}\label{GpEp}
index($A$) of $A \in \rnn$ denotes the smallest nonnegative integer $i$  such 
that $\rank(A^{i}) = \rank(A^{i+1})$ \cite[Definition 7.2.1]{CM} 
and it is thus equal to the size of the largest Jordan block 
corresponding to the zero eigenvalue of $A$ \cite[section 3]{OL}. 
If $\ran(A) \cap \nul(A) = {\vector{0}}$,
 $A$ is called a GP (group) matrix  \cite[section 1]{HaSp}. 
(A group matrix is a matrix which has a group inverse.)  
If $A$ is singular, $A$ is GP if and only if ${\rm index}(A) = 1$.
The GP matrix arises, for example, from the finite difference discretization of a convection diffusion 
equation with Neumann boundary condition 
\begin{eqnarray}\label{neueq}
\Delta u + d \frac{\partial u}{\partial x_{1}} & = & x_{1} + x_{2},~~\vector{x}=(x_{1}, x_{2}) \in \Omega \equiv [0,1] \times [0,1],  \\
\frac {\partial u(\vector{x})}{\partial \nu} & = & 0 ~~{\rm for}~~ \vector{x} \in \partial \Omega ,
\end{eqnarray}
as in \cite{BW}.

Assume $A$ arises from the finite difference discretization of the above equation. 
Then, $\nul(A^{\rm T}) \neq \nul(A)$ and 
$\ran(A) \cap \nul(A) = \{\vector{0}\}$ hold. Then, A is a GP matrix.

GP matrices also arise in the analysis of ergodic homogeneous finite Markov chains \cite{FH}.

When $\ran(A^{\rm T}) = \ran(A)$, $A$ is called range-symmetric or EP (Equal Projectors) \cite{CM}. Note that, since
$\nul(A^{\rm T}) = \ran(A)^{\perp}$, $\ran(A^{\rm T}) = \ran(A)
\Leftrightarrow \nul(A^{\rm T}) = \nul(A) \Leftrightarrow \ran(A) \perp \nul(A)$.
Hence, $\ran(A^{\rm T}) = \ran(A) \Rightarrow \ran(A) \cap \nul(A) 
= \{\vector{0}\}$.
That is, an EP matrix is a GP matrix. 
Therefore, if $A$ is EP and singular, then ${\rm index}(A) = 1$.
Furthermore, if ${\rm index}(A) \geq 2$, $A$ is not EP, 
that is $\ran(A) \neq \ran(A^{\rm T})$ holds.

An EP matrix arises, for instance, from the finite difference
approximation
of the convection diffusion equation (\ref{neueq}) with a periodic boundary condition 
\begin{eqnarray*}
u(x_{1}, 0) = u(x_{1}, 1), ~~x_{1} \in [0,1], \\
u(0, x_{2}) = u(1, x_{2}), ~~x_{2} \in [0,1],
\end{eqnarray*}
as in \cite{BW}.

\section{Convergence theory of GMRES for singular systems}\label{ConvGm}
Consider the system of linear equations (\ref{lineq})
and the least squares problem (\ref{lstsq}).
(\ref{lineq}) is called consistent 
when $\vector{b}\in \ran(A)$, and inconsistent otherwise.
Brown and Walker \cite{BW} showed that GMRES determines a least squares solution of 
(\ref{lstsq}) without breakdown for arbitrary $ \bb, \bx_0 \in \rn $ if and only if $A$ is range symmetric (EP), 
i.e. $\ran(A^{\rm T}) = \ran(A)$, where $ A \in \rnn $ may be singular and $\bb $ may not be  in the range space $\ran(A)$ of $A$.
In \cite{BW}, it was also pointed out that even if $\ran(A^{\rm T}) = \ran(A)$, if $\vector{b} \notin \ran(A)$,
the least squares problem (\ref{GMRES}) becomes very ill-conditioned. (The condition number of the Hessenberg matrix
arising in each iteration of GMRES becomes very large.)

If $A$ is a GP matrix, GMRES determines a solution of (\ref{lineq}) for all $\vector{b}\in \ran(A)$ and 
for all initial vector $\vector{x}_{0}$ (\cite[Theorem 2.6]{BW}). Moreover,  
GMRES determines a solution of (\ref{lineq}) for all $\vector{b}\in \ran(A)$ and 
for all initial vector $\vector{x}_{0}$ if and only if $A$ is a GP matrix (\cite{HS}, Theorem 2.8).
Even if $A$ is a GP matrix and $\vector{b} \in \ran(A)$, 
it was reported in \cite{MR} that 
the least squares problem (\ref{GMRES})
can become very ill-conditioned before breakdown in finite precision arithmetic.
Furthermore, if $A$ is GP and is not EP, 
there exists $\vector{b} \notin \ran(A)$ such that
GMRES breaks down at iteration step $1$ without giving a least squares 
solution of (\ref{lstsq}) (\cite[Theorem 2.6]{HS}).

If ${\rm index}(A) > 1$, there exists a $\vector{b} \in \ran(A)$ such that GMRES breaks down at iteration step 1 without
giving a solution of (\ref{lineq}) (\cite[Theorem 2.8]{HS}). 

In general,
when $\ran(A^{\rm T}) \ne \ran(A)$, there exist $\vector{x}_{0}$ and 
$\vector{b}$ such that 
GMRES breaks down without giving a least squares solution of (\ref{lstsq}) (\cite{BW, HS}).

In this paper, we propose using right preconditioning 
to overcome the difficulty of solving the least squares problems 
(\ref{lstsq}) with coefficient matrices 
which are not EP, including the case when the
index is greater than or equal to 1 \cite{SH}.
Furthermore, we propose using RRGMRES instead of GMRES since RRGMRES is more stable compared to GMRES
for inconsistent range-symmetric systems.
Finally, we propose applying the NR-SSOR inner iteration as the right preconditioner.
The proposed method also works well for (underdetermined) least squares problems.

RRGMRES determines a least squares solution 
of $A\vector{x} = \vector{b}$
without breakdown for arbitrary $\vector{b} \in \rn$,
and initial iterate $\vector{x}_0 \in \rn$
where $A$ may be singular 
if $\ran(A) = \ran(A^{\rm {T}})$, and furthermore this 
least squares solution is also the minimum-norm solution if
$\vector{x}_0 \in \ran(A)$. \cite{CLR}. 
We also describe the convergence theory for RRGMRES in Appendix \ref{rrnonsg} and \ref{rrsg}.

\section{Right preconditioned GMRES and RRGMRES for arbitrary singular systems
and rectangular least squares problems}\label{RpreGrr}
For the least squares problem (\ref{lstsq2}) 
where $ A \in \rmn $, Hayami, Yin and Ito \cite{HYI} proposed the AB-GMRES method which applies GMRES to (\ref{AB})
where
$ B \in \rnm $ satisfies $ \ran (AB) = \ran (A) $.
Note that 
$\displaystyle  \min_{\bz \in \Rm} \| \bb -AB \bz \|_2 = \min_{\bx \in \rn} \| \bb -A \bx \|_2$  holds
for all $\vector{b} \in \Rm$ if and only if $ \ran(A) = \ran (AB) $ (\cite{HYI}, Theorem 3.1).
AB-GMRES is the right preconditioned GMRES using $B$ as the right preconditioner.
If $B=CA^{\rm T} $, where $ C \in \rnn $ is symmetric positive definite, then $ \ran(A) = \ran (AB) $ 
holds (\cite{HYI}, Lemma 3.4).
Note also that $\ran ((AB)^{\rm T}) = \ran(AB) $, i.e. $ AB $ becomes 
range-symmetric, so that the AB-GMRES determines a least 
squares solution of (\ref{AB}) without breakdown for arbitrary $ \bb \in \Rm$ \cite{SH}. 

Note also that $B=CA^{\rm T}$ can be regarded as a right preconditioner.
Using this right preconditioning, we transform (\ref{lstsq}), where index($A$) is greater than or equal to 1 for $A \in \rnn$, 
to a least squares problem (\ref{AB}) with $ B=CA^{\rm T} $, where 
$ C \in \rnn $ is symmetric positive definite.
Since $\ran(A) = \ran (AB)$ holds
and $AB$ is range-symmetric, AB-GMRES determines a least squares solution 
$\vector{z}$ of (\ref{AB}) without breakdown
and $\vector{x} = B\vector{z}$ is a least squares solution of ($\ref{lstsq2}$).
However, when $ \bb \notin \ran(A) $, even when $\ran ( A^{\rm T} ) = \ran(A)$, the least squares problem (\ref{GMRES}) in GMRES may 
become dangerously ill-conditioned before a least squares solution of  (\ref{lstsq}) is obtained (\cite[Theorem 2.5]{BW}).
To overcome this problem, we proposed using pseudoinverse and used pinv in MATLAB 
for computing the pseudoinverse in \cite{SHL}, and further proposed using pinv with a larger
and appropriate threshold than the default value in MATLAB in \cite{SH}.
However, the computational cost of pinv is large. 
Thus, in this paper, we propose using RRGMRES instead of GMRES
since RRGMRES is more stable compared to GMRES
for inconsistent range-symmetric systems. 
This is because the smallest singular value of the Hessenberg matrix $H_{k+1,k}$
arising in RRGMRES at the $k$th step is bounded below by the 
$r$th 
largest singular value of 
the coefficient matrix $A$ with $\rank A = r$ as shown in (\ref{sgHbd}) 
of Appendix \ref{rrsg}. See also \cite{MR}.
This does not hold with GMRES \cite{BW, MR}.
The computational cost of RRGMRES 
is almost the same as that of GMRES. RRGMRES converges to a least squares solution for range-symmetric systems.
Furthermore, when the initial iterate $\vector{x}_0 \in \ran(A)$ holds,
this solution is a minimum norm solution. 

Using $B = CA^{\rm T}$ where $C$ is symmetric positive definite as a right 
preconditioner and $A \in \rmn$,
the right preconditioned RRGMRES also determines a least 
squares solution of (\ref{lstsq2}) without breakdown 
for arbitrary $ \bb \in \Rm$ and arbitrary $\vector{x}_0 \in \rn$
since $AB$ is range-symmetric and RRGMRES determines a least squares solution
for range-symmetric systems.

\section{NR-SSOR right preconditioned RRGMRES and its application to 
singular systems and rectangular least squares problems}\label{NRssApp}

We proposed using $C = I$ and $\{{\rm diag}(A^{\rm T} A)\}^{-1}$ as the right preconditioner $B = CA^{\rm T}$ in \cite{SH}.
In this section, we further propose using NR-SSOR inner iterations \cite{MH} as a more powerful right preconditioner 
and show that it can also be used for least squares problem (\ref{lstsq2}) with $A \in \rmn$.
Then, we prove that for the NR-SSOR preconditioner $B = CA^{\rm T}$, 
$C \in \rnn$ is symmetric positive definite.
Furthermore, we consider the case when $m = n$ or $m \neq n$ for $A \in \rmn$.

\subsection{NR-SSOR preconditioner}
NR-SSOR is mathematically equivalent to SSOR applied to
$A^{\rm T}A\vector{x} = A^{\rm T}\vector{b}$.
Then,
consider the first kind normal equations
\begin{equation}\label{fknq}
A^{\rm T}A\vector{x} = A^{\rm T}\vector{b}
\end{equation}
where $A \in \rmn$ and $\vector{b} \in \Rm$,
which is equivalent to the least squares problem (\ref{lstsq2}).

Let
\begin{equation}\label{ssform}
A^{\rm T}A  =  M - N \in \rnn
\end{equation} 
where $M$ is nonsingular, and 
\begin{eqnarray*}
H & = & M^{-1}N \\
  &=  & I - M^{-1}A^{\rm T}A
\end{eqnarray*}
Now consider the stationary iterative method applied to (\ref{fknq})
with $\vector{x}^{(0)} = \vector{0}$.
\begin{eqnarray*}
\vector{x}^{(\ell)} & = & H\vector{x}^{(\ell-1)} + M^{-1}A^{\rm T}\vector{b}\\
                 & = & H^{\ell}\vector{x}^{(0)} + \sum_{i=0}^{\ell-1}H^{i}M^{-1}A^{\rm T}\vector{b}\\
                 & = & \sum_{i=0}^{\ell-1}H^{i}M^{-1}A^{\rm T}\vector{b}
\end{eqnarray*}
Then, let $B^{(\ell)} = C^{(\ell)}A^{\rm T}$ be the preconditioner where
\begin{equation*}
C^{(\ell)} = \sum_{i=0}^{\ell-1}H^{i}M^{-1},
\end{equation*}
and
$\displaystyle H^{i}M^{-1}  =  (I - M^{-1}A^{\rm T}A)^{i}M^{-1}$.

If $M$ is symmetric, $\displaystyle H^{i}M^{-1}  =  (I - M^{-1}A^{\rm T}A)^{i}M^{-1}$ is symmetric. 
Therefore, if $M$ is symmetric, $C^{(\ell)}$ is symmetric.

We give the algorithm of NR-SSOR for the inner-iterations preconditioning
which works on the normal equations $A^{\rm T}A\vector{z} = A^{\rm T}\vector{c}$. 
Let $\vector{a}_j$ be the $j$th column of $A$.
\\ \\
{\bf Algorithm 3: NR-SSOR}\\ \\
Let $\vector{z}^{(0)} : = \vector{0}$ and $\vector{r} : = \vector{c}$. \\
For $ k = 1, 2, \cdots ,\ell,$ Do \\
\hspace{5mm}For $ j = 1, 2, \cdots ,n,$ Do \\
\hspace{10mm} $ d_j^{(k-\frac{1}{2})} := \omega(\vector{r},\vector{a}_j )/\|\vector{a}_j\|_2^2$,
$z_j^{(k-\frac{1}{2})} := z_j^{(k-1)} + d_j^{(k-\frac{1}{2})}$,
$\vector{r} := \vector{r} - d_j^{(k-\frac{1}{2})}\vector{a}_j$\\
\hspace{5mm}End do \\
\hspace{5mm}For $ j = n, n-1, \cdots ,1,$ Do \\
\hspace{10mm} $ d_j^{(k)} := \omega(\vector{r},\vector{a}_j )/\|\vector{a}_j\|_2^2$,
$z_j^{(k)} := z_j^{(k-\frac{1}{2})} + d_j^{(k)}$,
$\vector{r} := \vector{r} - d_j^{(k)}\vector{a}_j$\\
\hspace{5mm}End do \\
End do

For NR-SSOR,
\begin{equation*}
M = \omega^{-1}(2-\omega)^{-1}(D + \omega L)D^{-1}(D+\omega L^{\rm T}),
\end{equation*}
where $A^{\rm T}A =L + D + L^{\rm T}$,
$L$ is a strictly lower triangular matrix,
$D$ is a diagonal matrix, and $\omega$ is the relaxation parameter.
Thus, $M$ is symmetric. 

Here, we define the following \cite{MP}.
\begin{definition}\label{Semi}
A matrix $C$ is called semiconvergent if $\lim_{i \rightarrow \infty} C^i$ exists.
\end{definition}

Then, the iteration matrix $H = M^{-1}N$ for 
NR-SSOR with $0 < \omega < 2$ is semiconvergent \cite{DAX}. 
$M$ is nonsingular if $A$ has no zero columns. 
Furthermore, $M$ is positive definite if $0 < \omega < 2$.
Thus, $M$ is symmetric positive definite.
Therefore, $C^{(\ell)}$ is symmetric if $A$ has no zero columns 
and $0 < \omega < 2$ holds.

\subsection{Application to least squares problems}

The right preconditioned GMRES is applicable to both overdetermined and 
underdetermined problems. However the method is especially 
suited for underdetermined problems since GMRES works in a smaller dimension $m$ than $n$.

In this section, we prove that the corresponding $C$ for NR-SSOR right 
preconditioning is symmetric positive definite. Therefore, 
we demonstrate that NR-SSOR right preconditioned GMRES and RRGMRES 
determine a least squares solution $\vector{z}$ of (\ref{AB}) without breakdown
for arbitrary $\vector{b} \in \Rm$ and initial iterate $\vector{x}_0 \in \rn$,
and $\vector{x} = B\vector{z}$ is a least squares solution of ($\ref{lstsq2}$). 

Here, semiconvergence is algebraically characterized as follows.
\begin{theorem}(\cite{HEN}, \cite[Theorem 1]{OL}, 
\cite[Theorem 2]{TAN})\label{SemiAlg}
The following are equivalent:

\begin{enumerate}
\item $C$ is semiconvergent.
\item For any eigenvalue $\lambda$ of $C$, either
\begin{enumerate}
\item $|\lambda| < 1$ or
\item $\lambda = 1$ and ${\rm index}(I-C) =1$
\end{enumerate}
holds.
\end{enumerate}

\end{theorem}

Using NR-SSOR right preconditioning, the preconditioner is $B = CA^{\rm T}$
where $C$ is symmetric positive definite according to the following theorem.

\begin{theorem}\label{Cpos}
Assume that $A$ has no zero columns and $0 < \omega < 2$ holds.
Then, $\displaystyle C^{(\ell)} = \sum_{i=0}^{\ell-1}H^{i}M^{-1} $ is positive 
definite.
\end{theorem}

\begin{proof}
Let $J = S^{-1}(I-H)S$ be the Jordan canonical form of $I-H$.
Here, the iteration matrix $H = M^{-1}N$ for NR-SSOR with $0 < \omega < 2$
is semiconvergent when $A$ has no zero columns.
Then, ${\rm index}(I-H) = {\rm index}(J) \leq 1$.
Without loss of generality, we denote
$J$ by $J ={\rm diag}(\tilde{J}, 0_{n-r})$, where $r =\rank A$,
$\tilde{J}$ has no eigenvalues equal to zero, and $0_{n-r}$ is a zero matrix of size
$n-r$. 
Here, the eigenvalues of $H = I - M^{-1}A^{\rm T}A$ are real numbers
since $M$ is symmetric positive definite when 
$A$ has no zero columns and $0 < \omega < 2$.
Let $\nu$ be an eigenvalue of $H$ such that $|\nu| < 1$, and 
let $\lambda = 1 - \nu$ be the corresponding nonzero eigenvalue of $I-H$. 
Then, $\lambda > 0$.
The corresponding eigenvalue of $I_{r} - (I_{r} - \tilde{J})^{\ell}$
is $\mu = 1 - (1 - \lambda)^{\ell}$ for all $\ell \geq 1$,
\begin{eqnarray*}
\mu & = & 1 - (1 - \lambda)^{\ell} \\
    & > & 1 - |1-\lambda|^{\ell} \\
    & > & 0
\end{eqnarray*}
since $|1 - \lambda| < 1$.

The eigenvectors of $[I_{r} - (I_{r} - \tilde{J})^{\ell}]$ are the same
as the eigenvectors of $\tilde{J}^{-1}$.
Further, the eigenvalues of $[I_{r} - (I_{r} - \tilde{J})^{\ell}]$ and the eigenvalues $\tilde{J}^{-1}$ are positive. 
Then, all eigenvalues of $[I_{r} - (I_{r} - \tilde{J})^{\ell}] \tilde{J}^{-1}$ are positive.  

Thus, 
all eigenvalues of $S {\rm diag}\{[I_{r} - (I_{r} - \tilde{J})^{\ell}\}\tilde{J}^{-1}, {\ell} I_{n-r}\}S^{-1}$ 
are positive.

We define $P \in \rnn$ as follows.
\begin{eqnarray*}
P & = & \sum_{i=0}^{\ell-1} H^{i} \\
  & = & \sum_{i=0}^{\ell-1} (I - M^{-1}A^{\rm T}A)^{i} 
\end{eqnarray*} 
Here,
$P = S {\rm diag}\{[I_{r} - (I_{r} - \tilde{J})^{\ell}\}\tilde{J}^{-1}, {\ell} I_{n-r}\}S^{-1}$.

Since $\displaystyle M^{\frac{1}{2}}PM^{- \frac{1}{2}}$ is symmetric,
$\displaystyle M^{\frac{1}{2}}PM^{- \frac{1}{2}} = V \Sigma V^{\rm T}$
where $V \in \rnn$ is an orthogonal matrix and $\Sigma \in \rnn$ is a diagonal matrix.
Furthermore, the diagonal element of $\Sigma$ are the eigenvalues of 
$P$.

Here, all eigenvalues of $P$ are positive. 
Thus, all diagonal elements of $\Sigma$ are positive. 

For $\forall \vector{y} \in \rn$,
\begin{eqnarray*}
(PM^{-1}\vector{y}, \vector{y}) 
& = & (M^{- \frac{1}{2}}V \Sigma V^{\rm T} M^{- \frac{1}{2}} \vector{y}, \vector{y}) \\
              & = & (\Sigma V^{\rm T} M^{- \frac{1}{2}} \vector{y}, V^{\rm T} M^{- \frac{1}{2}}\vector{y}).
\end{eqnarray*}

Here, 
$\vector{v} = (v_{1}, v_{2},...,v_{n}) \in \rn$ is defined as follows.
\begin{equation*}
\vector{v} = V^{\rm T} M^{- \frac{1}{2}} \vector{y}.
\end{equation*}
Furthermore, $\sigma_{i}$ is defined as the $(i,i)$ element of $\Sigma$.

Then,
\begin{eqnarray*}
(PM^{-1}\vector{y}, \vector{y}) & = & 
(\Sigma V^{\rm T} M^{- \frac{1}{2}} \vector{y}, V^{\rm T} M^{- \frac{1}{2}}\vector{y}) \\
              & = & \sum_{i=1}^{n} \sigma_{i} v_{i}^{2}.
\end{eqnarray*}
Since $\sigma_{i}(1 \leq i \leq n)$ is positive and $v_{i} (1 \leq i \leq n)$ are real numbers,
$(PM^{-1}\vector{y}, \vector{y})$ is positive.
Since $PM^{-1}$ is symmetric,
$PM^{-1}$ is positive definite.

Since 
\begin{eqnarray*}
C^{(\ell)} & = & S {\rm diag}\{[I_{r} - (I_{r} - \tilde{J})^{\ell}\}\tilde{J}^{-1}, {\ell} I_{n-r}\}S^{-1}M^{-1},
\end{eqnarray*}
$C^{(\ell)} = PM^{-1}$ is positive definite.
\end{proof} 

Therefore, for the NR-SSOR right preconditioner $B = CA^{\rm T}$, $C \in \rnn$ is symmetric positive definite if $A$ has no zero columns 
and $0 < \omega < 2$ holds.
Then, since $AB$ is symmetric, $AB$ is range-symmetric.
On the other hand,
the right preconditioned GMRES with $B = A^{\rm T}C$ using stationary inner iterations converges to a solution 
of $A\vector{x}= \vector{b}$ for all $\vector{b} \in \ran(A)$ and all 
initial iterate $\vector{x}_0 \in \rn$ (\cite{MH}, Theorem 5.5, 5.6).
However, the right preconditioned GMRES and RRGMRES with $B = A^{\rm T}C$ 
using stationary inner iterations do not necessarily 
determine a least 
squares solution of (\ref{AB}) for arbitrary $ \bb \in \Rm$.
since $AB$ is not necessarily range-symmetric.
On the other hand, the
NR-SSOR right preconditioned GMRES and RRGMRES determines a least 
squares solution of (\ref{AB}) for arbitrary $ \bb \in \Rm$ since $AB$ 
for NR-SSOR right preconditioning is range-symmetric.

Next, we will investigate the distribution of the eigenvalues of $AC^{(\ell)}A^{\rm T}$ 
where $C^{(\ell)}A^{\rm T}$ is the preconditioned matrix by $\ell$ NR-SSOR inner iterations.


According to \cite{DAX}, the iteration matrix $H$ for NR-SSOR with $0 < \omega < 2$ is semiconvergent.

Let $B^{(\ell)} = C^{(\ell)}A^{\rm T}$. Then, 
the following theorem holds \cite{MH}.

\begin{theorem}(\cite[Theorem 4.8]{MH})\label{MHeig}
Let $r = {\rm rank} A$. Assume that $H$ is semiconvergent. Then, there exists $r$ eigenvalues of 
of $B^{(\ell)}A$ in a disk with a center at $1$ and radius $\rho(H)^{\ell} < 1$, and
the remaining $n - r$ eigenvalues are zero.
\end{theorem}

The following lemma holds.

\begin{lemma}\label{ABeig}
Let $r = {\rm rank} A$.
Assume that $A$ has no zero columns and $0 < \omega < 2$ holds.
Let $\displaystyle C^{(\ell)} = \sum_{i=0}^{\ell-1}H^{i}M^{-1} $ for $\ell$ NR-SSOR inner iterations.
Then, $A C^{(\ell)}A^{\rm T}$ and $C^{(\ell)}A^{\rm T}A$ have the same $r$ 
nonzero 
eigenvalues which are real.
\end{lemma}

\begin{proof}
Let $\by \in \rn$ be an eigenvector corresponding to a nonzero eigenvalue $\lambda$ of $C^{(\ell)}A^{\rm T}A$.
That is, $\displaystyle C^{(\ell)}A^{\rm T}A\by = \lambda \by$.
Then, $\displaystyle A C^{(\ell)}A^{\rm T}A\by = \lambda A\by$.
Here, $A\by \in \Rm$ is a nonzero vector since $\lambda$ is nonzero.
Hence, $A\by \in \Rm$ is an eigenvector corresponding to a nonzero eigenvalue $\lambda$ of $A C^{(\ell)}A^{\rm T}$.

$C^{(\ell)}$ is a symmetric positive definite matrix. 
Define $\tilde{A} = A{C^{(\ell)}}^{\frac{1}{2}}$.
Then,\\ $A C^{(\ell)}A^{\rm T} = \tilde{A}{\tilde{A}}^{\rm T}$.
Let $\bz \in \Rm$ be an eigenvector corresponding to a nonzero eigenvalue $\nu$ of $AC^{(\ell)}A^{\rm T}$.
That is, $\displaystyle \tilde{A}{\tilde{A}}^{\rm T}\bz = \nu \bz$.
Then, $\displaystyle {\tilde{A}}^{\rm T}\tilde{A}{\tilde{A}}^{\rm T}\bz = \nu {\tilde{A}}^{\rm T}\bz$.
Here, ${\tilde{A}}^{\rm T}\bz \in \rn$ is a nonzero vector since $\nu$ is nonzero.
Let $\bw = {\tilde{A}}^{\rm T}\bz$. 
Then, $\displaystyle {\tilde{A}}^{\rm T}\tilde{A}\bw = \nu \bw$.
Here the eigenvalues of ${\tilde{A}}^{\rm T}\tilde{A} = {C^{(\ell)}}^{\frac{1}{2}}A^{\rm T}A{C^{(\ell)}}^{\frac{1}{2}}$
are equal to the eigenvalues of $C^{(\ell)}A^{\rm T}A$.
Then, the eigenvalue $\nu$ of $AC^{(\ell)}A^{\rm T}$ is an eigenvalue of $C^{(\ell)}A^{\rm T}A$.

Furthermore,  $C^{(\ell)}A^{\rm T}A$ has $r$ nonzero eigenvalues from Theorem \ref{MHeig} (\cite{MH}, Theorem 4.8).
Then, $A C^{(\ell)}A^{\rm T}$ and $C^{(\ell)}A^{\rm T}A$ have the same $r$
nonzero eigenvalues.   

Since $C^{(\ell)}$ is symmetric positive definite, all eigenvalues of $A C^{(\ell)}A^{\rm T}$ are real.
Then, all eigenvalues of $C^{(\ell)}A^{\rm T}A$ are also real.
\end{proof}

From Theorem \ref{MHeig} (\cite[Theorem 4.8]{MH}) and Lemma \ref{ABeig},
the following theorem holds.

\begin{theorem}\label{abeg}
Assume that $A$ has no zero columns and $0 < \omega < 2$ holds. Let $r = {\rm rank} A$,
$M = \omega^{-1}(2 - \omega)^{-1}(D + \omega L)D^{-1}(D + \omega L^{\rm T})$,
$H = I - M^{-1}A^{\rm T}A$ and
$\displaystyle C^{(\ell)} = \sum_{i=0}^{\ell-1}H^{i}M^{-1} $.

Then, when $\ell$ is even, there exist $r$ eigenvalues of 
$A C^{(\ell)}A^{\rm T}$
which are real numbers in the interval $[1 - \rho(H)^{\ell}, 1]$.

On the other hand,  when $\ell$ is odd, there exist $r$ eigenvalues of 
$A C^{(\ell)}A^{\rm T}$
which are real numbers in the interval $[1 - \rho(H)^{\ell}, 1 + \rho(H)^{\ell}]$.

For both cases when $\ell$ is even or odd,
the remaining $n - r$ eigenvalues are zero where $H = I - M^{-1}A^{\rm T}A$.
\end{theorem}

Theorem \ref{abeg} shows that if $A$ has no zero columns and $0 < \omega < 2$, 
then the $r$ nonzero eigenvalues of $A C^{(\ell)}A^{\rm T}$ approach 1 as $\ell$ increases.

\subsection{minimum-norm solution}
The right preconditioned GMRES and RRGMRES solve $AB \vector{u} = \vector{b}$ so that 
$\vector{x} = B\vector{u} = CA^{\rm T}\vector{u}$.
The necessary and sufficient condition that the solution $\vector{x}$ is a minimum-norm solution 
is $\vector{x} \in \nul(A)^{\perp} = \ran(A^{\rm T})$.
If $C = I$, $\vector{x} \in \ran(A^{\rm T})$ holds since 
$\vector{x} = CA^{\rm T}\vector{u}$,
but if $C \neq I$, $\vector{x} \in \ran(A^{\rm T})$ does not necessarily hold.
Therefore, if $C = I$, the solution of the right preconditioned GMRES and 
RRGMRES
is the minimum-norm solution, but if $C \neq I$, the solution is not necessarily the minimum-norm solution.

\section{MINRES-QLP}\label{MINQLP}
For indefinite or singular symmetric systems, MINRES-QLP \cite{CPS} was proposed
by improving the performance of MINRES. 
For singular symmetric systems, MINRES determines a least squares solution. Especially,
for singular symmetric consistent systems, 
MINRES determines a min-norm solution if the initial iterate $\vector{x}_0 \in \ran(A)$.
On the other hand,  
for singular symmetric systems, MINRES-QLP determines a min-norm solution 
even if the systems are inconsistent.
Therefore, MINRES-QLP is more robust than MINRES for singular symmetric systems.

The right preconditioned matrix $ACA^{\rm T}$ in section \ref{RpreGrr} is symmetric 
since $C$ is symmetric positive definite.
Therefore, for the case when $C$ is an identity matrix, 
we will apply MINRES-QLP to $AA^{\rm T}\vector{z} = \vector{b}$ and compute the solution
$\vector{x} = A^{\rm T}\vector{z}$ for $A\vector{x} = \vector{b}$, and 
compare it with AB-RRGMRES when $C$ is an identity matrix in $B = CA^{\rm T} = A^{\rm T}$.

\section{Numerical experiments}
In the numerical experiments for inconsistent problems with GP and index 2 coefficient matrices 
$A$, we are mainly concerned with 
verifying
how small the residual of the proposed method can become.
The estimation of the cost performance 
of the proposed method is left for future work.
Therefore, the size of the test matrices in the numerical experiments for GP and index 2 inconsistent problems were kept small.

On the other hand, the size of the test matrices in the numerical experiments for underdetermined least squares 
problems is not small.

The numerical experiments of AB-RRGMRES, AB-GMRES 
and MINRES-QLP
were performed on 
a PC with 11th Gen Intel(R) Core(TM) i7-1165G7 2.80 GHz CPU, Ubuntu 20.04 LTS 
and double precision
floating arithmetic. 
All programs for the iterative methods in our tests were coded in Fortran 90
and were compiled by gfortran. Here, the version of gcc is 9.4.0.
For computing singular value decomposition of matrices, we used MATLAB R2018b.
We used the code from \cite{QLP} for the MINRES-QLP.

\subsection{GP inconsistent problems}\label{NumGP}

The GP matrix $A \in {\bf R}^{128 \times 128}$ is as follows.
\begin{eqnarray*}
\left[
\begin{array}{cc}
A_{11} & A_{12} \\
0                                 &  0                            
\end{array}
\right],
\end{eqnarray*}
where $A_{11}, A_{12} \in {\bf R}^{64 \times 64}$.
Here, assume $J_{k}(\lambda) \in {\bf R}^{k \times k} $ is a square matrix of the form
\begin{eqnarray*}
\left[
\begin{array}{cccc}
\lambda & 1 & 0 & 0\\
0      & \lambda & 1 & 0 \\
0      & 0      & ... & 1 \\
0      &  0     & 0   & \lambda\\                            
\end{array}
\right]
\end{eqnarray*}
where $\lambda \in \real$.

\begin{eqnarray*}
A_{11} & = & \left[
\begin{array}{cc}
W & 0 \\
0                                 &  D                            
\end{array}
\right],
\end{eqnarray*}
where $W \in {\bf R}^{32 \times 32}$ is 
\begin{eqnarray*}
\left[
\begin{array}{cccc}
J_{2}(\alpha_{1}) & 0 & 0 & 0 \\
0                 &  J_{2}(\alpha_{2}) & 0 & 0 \\
0                 &        0           & ... & 0 \\  
0                 &        0           & 0 & J_{2}(\alpha_{16})  
\end{array}
\right]
\end{eqnarray*}
and $D$ is a diagonal matrix whose $j$th diagonal element is $\beta_{j}$.

Here, $\alpha_{j}(j=1,2,...16) \in \real$, $\beta_{i}(i=1,2,...,32) \in \real$ are
as follows, respectively 
\begin{eqnarray}
\alpha_{1} = 1 , \alpha_{16} = 10^{-\rho},~ \alpha_{j} 
= \alpha_{16} + \frac{16-j}{15}(\alpha_{1} - \alpha_{16})\times 0.7^{j-1} \label{alpeq} \\
\beta_{1} = 1 , \beta_{32} = 10^{-\gamma},~ \beta_{i} 
= \beta_{32} + \frac{32-i}{31}(\beta_{1} - \beta_{32})\times 0.2^{i-1} \label{beteq}
\end{eqnarray}
Furthermore, $A_{12} \in {\bf R}^{64 \times 64}$ is
\begin{eqnarray*}
\left[
\begin{array}{cccc}
J_{2}(\beta_{1}) & 0 & 0 & 0 \\
0                 &  J_{2}(\beta_{2}) & 0 & 0 \\
0                 &        0           & ... & 0 \\  
0                 &        0           & 0 & J_{2}(\beta_{32})  
\end{array}
\right]
\end{eqnarray*}

For the above matrix $A$, let $C = \{{\rm diag}(A^{{\rm T}}A)\}^{-1} \in \rnn$.
Then, $C$ is symmetric positive definite.
In this section, let both $\rho = \gamma = 12$.

Table \ref{gpinfo} gives information on this GP matrix.
The condition number was computed by dividing the largest singular value by the smallest nonzero
one, where the singular values were determined by using the MATLAB function {\bf svd} and 
the number of nonzero singular values was determined by using the MATLAB function {\bf rank}.
In Table \ref{gpinfo}, Rank is the maximum number of linearly independent columns of the matrix $A$
and $\kappa(A)$ is the condition number $\sigma_1/\sigma_r$ of the matrix, 
where $\sigma_1$ and $\sigma_r$ are the largest and $r$th largest singular values of the matrix, respectively, 
and $r$ is the rank.

\begin{table}[htbp]
\centering
\caption{Characteristics of the coefficient matrices of the GP problems}
\label{gpinfo}
\begin{tabular}{|r|r|}
\hline
Rank & $\kappa(A)$   \\
\hline
\hline
64 & {\rm $2.29 \times 10^{12}$}  \\
\hline
\end{tabular}
\end{table}

The right hand side vector $\vector{b}$ was set as 
$\vector{b} = \frac{A \times (1,1,.,1)^{{\rm T}}}{\|A \times (1,1,...,1)^{{\rm T}}\|_{2}}
 + \frac{\vector{u}(0,1)}{\|\vector{u}(0,1)\|_{2}}\times 0.01$
, where $\vector{u}(0,1)$ 
is an $n$ dimensional vector of pseudorandom numbers generated according to the uniform 
distribution in the interval [0,1].
Thus, the systems are generically inconsistent.



For AB-RRGMRES using NR-SSOR,
we determined the number of inner iterations $\ell$
as the smallest 
among 1,2,4,6,8,10 such that
the minimum value of
$\displaystyle 
\frac{\|A^{{\rm T}}\vector{r}_{k}\|_{2}}{\|A^{{\rm T}}\vector{b}\|_{2}}$ becomes smaller than $10^{-13}$
when the relaxation parameter $\omega $ is $1.0$.
Next, fixing $\ell$, we determined the optimal $\omega$ so that the method minimizes
$\displaystyle 
\frac{\|A^{{\rm T}}\vector{r}_{k}\|_{2}}{\|A^{{\rm T}}\vector{b}\|_{2}}$
among $\omega = 1.8,1.6,1.4,1.2,1.1,1.0,0.9,0.8,0.6,0.4,0.2$.
The parameters for AB-GMRES using NR-SSOR were determined similarly.
As a result, we chose $\omega = 1.0$ and $\ell = 1$
for AB-RRGMRES and $\omega = 0.9$ and $\ell = 1$
for AB-GMRES.

Fig. \ref{ag_i} 
shows $\displaystyle \frac{\|A^{{\rm T}}\vector{r}_{k}\|_{2}}{\|A^{{\rm T}}\vector{b}\|_{2}}$ versus 
the number of iterations $k$ for RRGMRES and GMRES.
applied to $A\vector{x}=\vector{b}$ directly without any preconditioning.

Fig. \ref{acatg_i}
shows $\displaystyle \frac{\|A^{{\rm T}}\vector{r}_{k}\|_{2}}{\|A^{{\rm T}}\vector{b}\|_{2}}$ versus 
the number of (outer) iterations $k$ for AB-RRGMRES using NR-SSOR with 1 inner iteration
and the relaxation parameter $\omega = 1.0$, 
$B=\{{\rm diag}(A^{\rm T}A)\}^{-1} A$ 
and $B = A^{\rm T}$, respectively. 

Fig. \ref{abgmrrg_i}
shows $\displaystyle \frac{\|A^{{\rm T}}\vector{r}_{k}\|_{2}}{\|A^{{\rm T}}\vector{b}\|_{2}}$ versus 
the number of iterations $k$ for AB-RRGMRES using NR-SSOR with 1 inner iteration
and the relaxation parameter $\omega = 1.0$
and AB-GMRES using NR-SSOR with 1 inner iteration
and $\omega = 0.9$.

\begin{figure}[htbp]
\begin{center}
\includegraphics[scale=0.55]{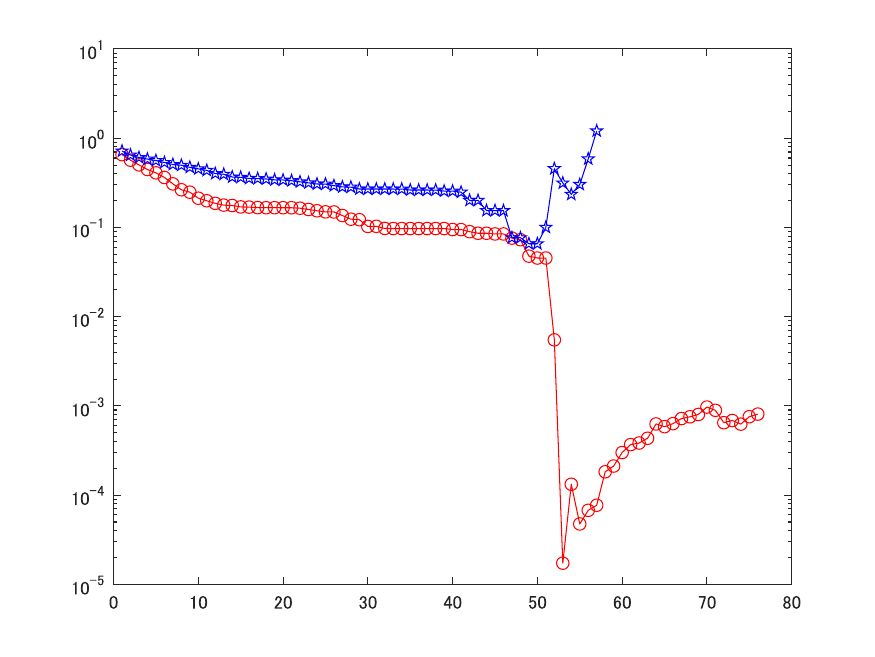} 
\caption{$\displaystyle \frac{\|A^{{\rm T}}\vector{r}_{k}\|_{2}}{\|A^{{\rm T}}\vector{b}\|_{2}}$ versus 
the number of iterations for RRGMRES (blue, $\star$) and 
GMRES (red, $\circ$) without preconditioning for the GP inconsistent system}
\label{ag_i}
\end{center}
\end{figure}

\begin{figure}[htbp]
\begin{center}
\includegraphics[scale=0.55]{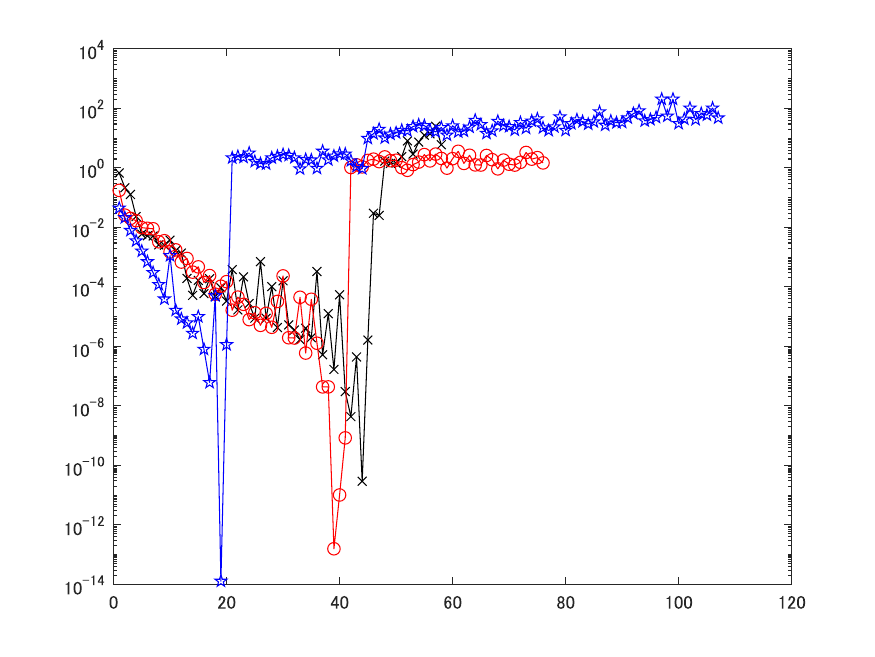} 
\caption{$\displaystyle \frac{\|A^{{\rm T}}\vector{r}_{k}\|_{2}}{\|A^{{\rm T}}\vector{b}\|_{2}}$ versus 
the number of (outer) iterations for  AB-RRGMRES using NR-SSOR
with 1 inner iteration and $\omega =1.0$ (blue, $\star$), 
$B=\{{\rm diag}(A^{\rm T}A)\}^{-1} A^{\rm T}$ (red, $\circ$), 
and $B=A^{\rm T}$ (black, $\times$) for the GP inconsistent system}
\label{acatg_i}
\end{center}
\end{figure}

\begin{figure}[htbp]
\begin{center}
\includegraphics[scale=0.55]{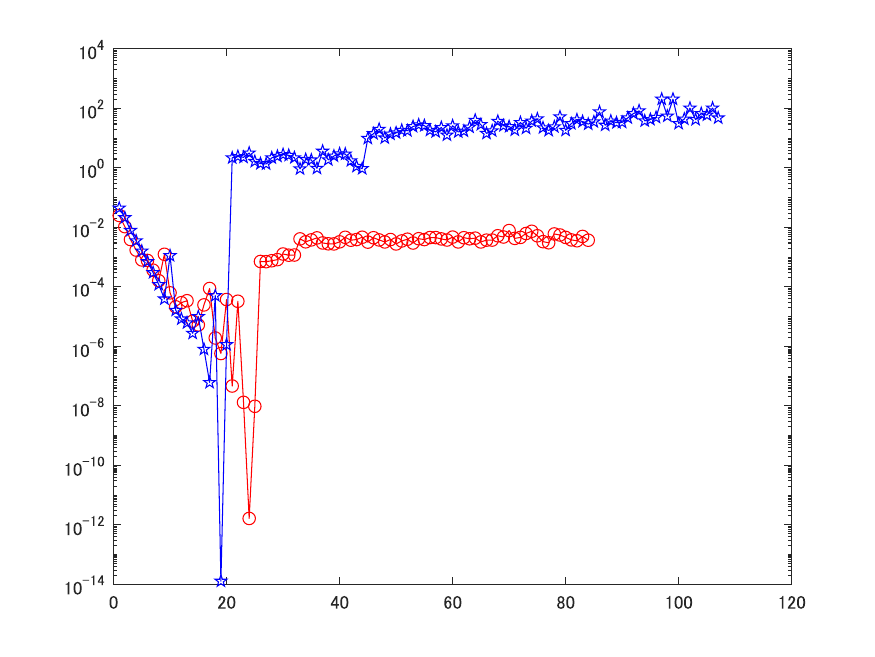}
\caption{$\displaystyle \frac{\|A^{{\rm T}}\vector{r}_{k}\|_{2}}{\|A^{{\rm T}}\vector{b}\|_{2}}$ versus 
the number of iterations for AB-RRGMRES using NR-SSOR
with 1 inner iteration and $\omega =1.0$ (blue, $\star$) and 
AB-GMRES using NR-SSOR
with 1 inner iteration and $\omega =0.9$ (red, $\circ$) 
for the GP inconsistent system}
\label{abgmrrg_i}
\end{center}
\end{figure}

We observe the following from Fig. \ref{ag_i}, \ref{acatg_i} and
\ref{abgmrrg_i}.

\begin{itemize}

\item Since the minimum value of 
$\frac{\|A^{\rm T}\vector{r}_{k}\|_{2}}{\|A^{\rm T}\vector{b}\|_{2}}$ for 
RRGMRES is larger 
than $10^{-2}$, RRGMRES does not converge to a least squares solution. The minimum value of 
$\frac{\|A^{\rm T}\vector{r}_{k}\|_{2}}{\|A^{\rm T}\vector{b}\|_{2}}$ for GMRES is more than $10^{-5}$, 
which is almost $10^{-4}$ times that of RRGMRES, 
but GMRES also does not converge to a least squares solution. 

\item The minimum value of $\frac{\|A^{\rm T}\vector{r}_{k}\|_{2}}{\|A^{\rm T}\vector{b}\|_{2}}$ for AB-RRGMRES 
using NR-SSOR with 1 inner iteration and $\omega =1$ 
is $10^{-14}$, which is almost $1/10$ of that using \\
$B=\{{\rm diag}(A^{\rm T}A)\}^{-1} A^{\rm T}$,
and $10^{-4}$ of $B = A^{\rm T}$.

\item The number of iterations when AB-RRGMRES using NR-SSOR with 1 inner iteration
and $\omega = 1.0$ attains the 
minimum value of 
$\frac{\|A^{\rm T}\vector{r}_{k}\|_{2}}{\|A^{\rm T}\vector{b}\|_{2}}$
is almost a half of that using $B=\{{\rm diag}(A^{\rm T}A)\}^{-1} A^{\rm T}$,
and $B = A^{\rm T}$

\item The minimum value of 
$\frac{\|A^{\rm T}\vector{r}_{k}\|_{2}}{\|A^{\rm T}\vector{b}\|_{2}}$ for 
AB-RRGMRES 
using NR-SSOR with 1 inner iteration and $\omega = 1$ is $10^{-2}$ times 
that of AB-GMRES using NR-SSOR with 1 inner iteration and $\omega = 0.9$.

\end{itemize}

\subsection{index 2 inconsistent problems}\label{NumId2}
The index $2$ matrix $A \in {\bf R}^{128 \times 128}$ is as follows.
\begin{eqnarray*}
\left[
\begin{array}{cc}
A_{11} & A_{12} \\
0                                 &  A_{22}                            
\end{array}
\right].
\end{eqnarray*}
where $A_{11}, A_{12}, A_{22} \in {\bf R}^{64 \times 64}$.
Here, $A_{11}, A_{12}, J_{k}(\lambda), \alpha, \beta$ are the same as those in Section \ref{NumGP}.
The $(2i -1 ,2i)(1 \leq i \leq 16)$ element of $A_{22}$ is $1$, and 
other elements are $0$. Thus, $A$ is an index $2$ matrix.

For the above matrix $A$, let $C \in \rnn$ be $C = \{{\rm diag}(A^{{\rm T}}A)\}^{-1}$.
Then, $C$ is symmetric positive definite.
In this section, let $\rho$ in (\ref{alpeq}) be $12$ and $\gamma$ in (\ref{beteq}) be $15$.

Table \ref{id2info} gives information on this index 2 matrix.
In Table \ref{id2info}, the definition of Rank and $\kappa(A)$ is the same as in Table \ref{gpinfo}.

\begin{table}[htbp]
\centering
\caption{Characteristics of the coefficient matrices of the index 2 problems}
\label{id2info}
\begin{tabular}{|r|r|}
\hline
Rank & $\kappa(A)$   \\
\hline
\hline
72 & {\rm $ 4.01 \times 10^{12}$}  \\
\hline
\end{tabular}
\end{table}

As with the GP inconsistent systems, the right-hand side vector $\vector{b}$ was set as \\
$\vector{b} = \frac{A \times (1,1,.,1)^{{\rm T}}}{\|A \times (1,1,...,1)^{{\rm T}}\|_{2}}
 + \frac{\vector{u}(0,1)}{\|\vector{u}(0,1)\|_{2}}\times 0.01$,
where $\vector{u}(0,1)$ 
is an $n$ dimensional vector of pseudorandom numbers generated according to the uniform 
distribution in the interval [0,1].


The optimal parameters were determined similarly to the GP problem in section \ref{NumGP}.
They are $\omega = 1.0$ and $\ell = 1$ for AB-RRGMRES using NR-SSOR,
and $\omega = 1.1$ and $\ell = 1$ for AB-GMRES using NR-SSOR.

Fig. \ref{aid2_i} 
shows $\displaystyle \frac{\|A^{{\rm T}}\vector{r}_{k}\|_{2}}{\|A^{{\rm T}}\vector{b}\|_{2}}$ versus 
the number of iterations $k$ for applying RRGMRES and GMRES directly to $A\vector{x} = \vector{b}$
without any preconditioning.
Fig. \ref{acatid2_i}
shows $\displaystyle \frac{\|A^{{\rm T}}\vector{r}_{k}\|_{2}}{\|A^{{\rm T}}\vector{b}\|_{2}}$ versus 
the number of iterations $k$ for AB-RRGMRES using NR-SSOR with 1 inner iteration and $\omega = 1.0$, 
$B=\{{\rm diag}(A^{\rm T}A)\}^{-1} A^{\rm T}$ 
and $B = A^{\rm T}$, respectively. 

Fig. \ref{abgmrrid2_i}
shows $\displaystyle \frac{\|A^{{\rm T}}\vector{r}_{k}\|_{2}}{\|A^{{\rm T}}\vector{b}\|_{2}}$ versus 
the number of iterations $k$ for AB-RRGMRES using NR-SSOR with 1 inner iteration
and $\omega = 1.0$
and AB-GMRES using NR-SSOR with 1 inner iteration
and $\omega = 1.1$.

\begin{figure}[htbp]
\begin{center}
\includegraphics[scale=0.55]{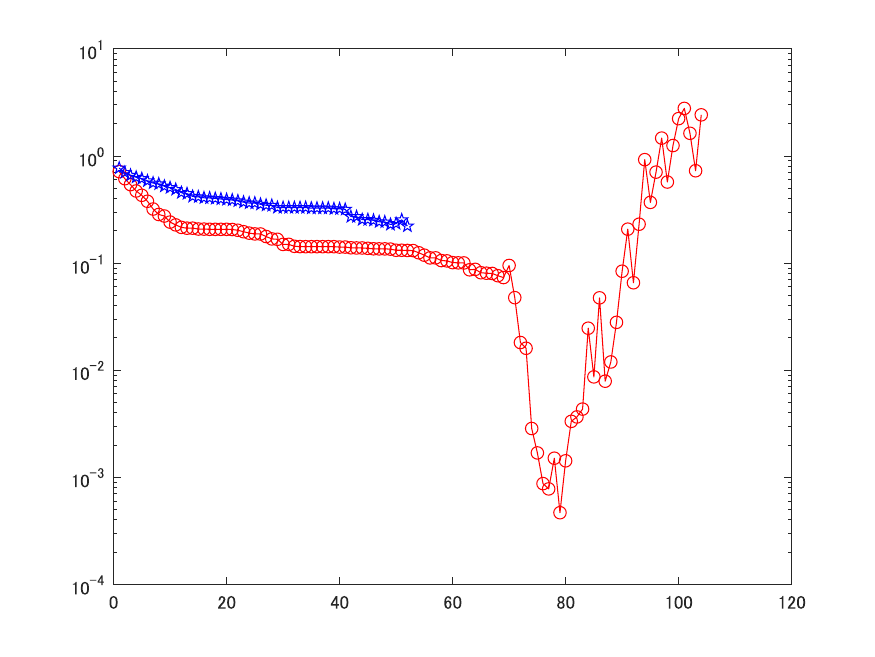}
\caption{$\displaystyle \frac{\|A^{{\rm T}}\vector{r}_{k}\|_{2}}{\|A^{{\rm T}}\vector{b}\|_{2}}$ versus 
the number of iterations for RRGMRES (blue, $\star$) and 
GMRES (red, $\circ$) without preconditioning for the index 2 inconsistent system }
\label{aid2_i}
\end{center}
\end{figure}

\begin{figure}[htbp]
\begin{center}
\includegraphics[scale=0.55]{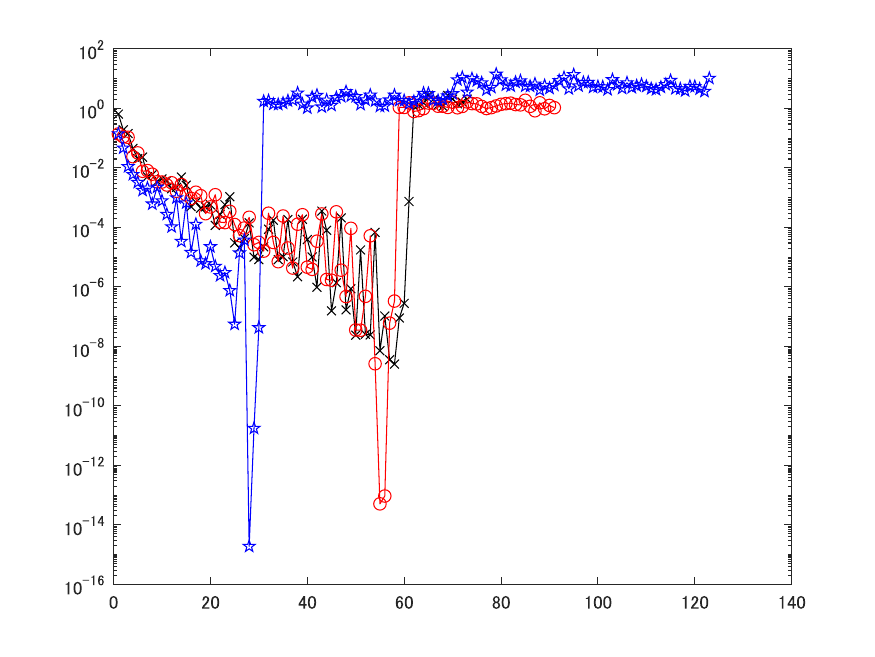}
\caption{$\displaystyle \frac{\|A^{{\rm T}}\vector{r}_{k}\|_{2}}{\|A^{{\rm T}}\vector{b}\|_{2}}$ versus 
the number of iterations for  AB-RRGMRES using NR-SSOR
with 1 inner iteration and $\omega = 1.0$ (blue, $\star$), 
$B=\{{\rm diag}(A^{\rm T}A)\}^{-1} A^{\rm T}$ (red, $\circ$), 
and $B=A^{\rm T}$ (black, $\times$)
for the index 2 inconsistent system}
\label{acatid2_i}
\end{center}
\end{figure}

\begin{figure}[htbp]
\begin{center}
\includegraphics[scale=0.55]{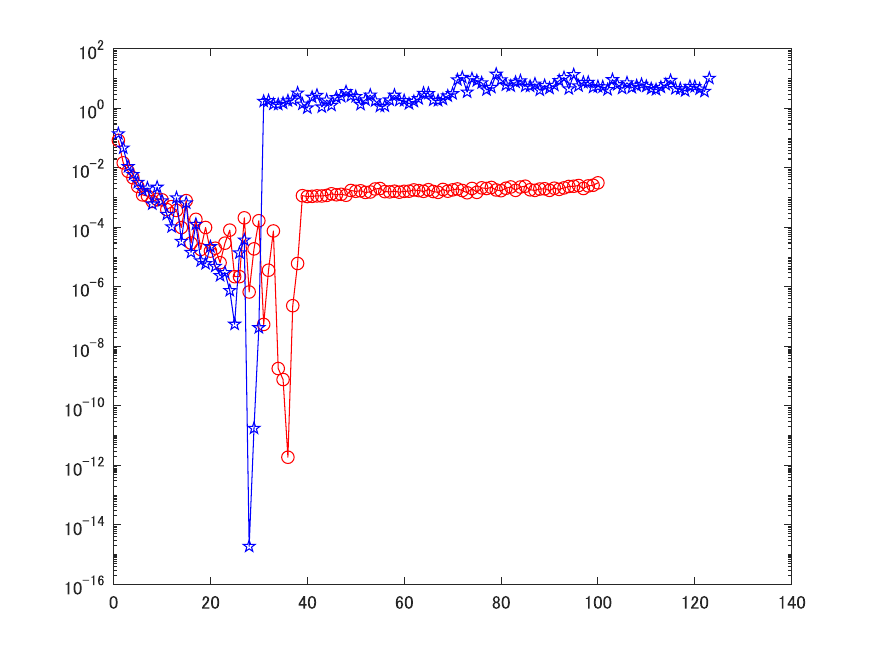}
\caption{$\displaystyle \frac{\|A^{{\rm T}}\vector{r}_{k}\|_{2}}{\|A^{{\rm T}}\vector{b}\|_{2}}$ versus 
the number of iterations for AB-RRGMRES using NR-SSOR
with 1 inner iteration and $\omega =1.0$ (blue, $\star$) and 
AB-GMRES using NR-SSOR
with 1 inner iteration and $\omega =1.1$ (red, $\circ$) 
for the index 2 inconsistent system}
\label{abgmrrid2_i}
\end{center}
\end{figure}

We observe the following from Fig. \ref{aid2_i}, \ref{acatid2_i} and \ref{abgmrrid2_i}.

\begin{itemize}

\item Since $\frac{\|A^{{\rm T}}\vector{r}_{k}\|_{2}}{\|A^{{\rm T}}\vector{b}\|_{2}}$ for RRGMRES is larger 
than $10^{-1}$, RRGMRES does not converge to a least squares solution. 
 The minimum value of 
$\frac{\|A^{\rm T}\vector{r}_{k}\|_{2}}{\|A^{\rm T}\vector{b}\|_{2}}$ for GMRES is more 
than $10^{-4}$, which is almost $10^{-3}$ times that of RRGMRES, 
but GMRES does not converge to a least squares solution.

\item The minimum value of $\frac{\|A^{{\rm T}}\vector{r}_{k}\|_{2}}{\|A^{{\rm T}}\vector{b}\|_{2}}$ 
for AB-RRGMRES 
using NR-SSOR with 1 inner iteration and $\omega = 1.0$ is almost $10^{-1}$ times smaller than 
with $B=\{{\rm diag}(A^{\rm T}A)\}^{-1} A^{\rm T}$, and
$10^{-7}$ of $B = A^{\rm T}$. 

\item The number of iterations of AB-RRGMRES using NR-SSOR with 1 inner iteration at the 
minimum value of $\frac{\|A^{{\rm T}}\vector{r}_{k}\|_{2}}{\|A^{{\rm T}}\vector{b}\|_{2}}$
is almost half of \\ $B=\{{\rm diag}(A^{\rm T}A)\}^{-1} A^{\rm T}$,
and $B = A^{\rm T}$.

\item The minimum value of 
$\frac{\|A^{\rm T}\vector{r}_{k}\|_{2}}{\|A^{\rm T}\vector{b}\|_{2}}$ for 
AB-RRGMRES 
using NR-SSOR with 1 inner iteration and $\omega = 1.0$ is $10^{-3}$ times 
that of AB-GMRES using NR-SSOR with 1 inner iteration and $\omega = 1.1$.

\end{itemize}

\subsection{Comparison with MINRES-QLP for GP and index 2 inconsistent problems}
For the GP inconsistent problem of section \ref{NumGP} and 
index 2 inconsistent problem of section \ref{NumId2}, 
we compare $\displaystyle \frac{\|A^{{\rm T}}\vector{r}_{k}\|_{2}}{\|A^{{\rm T}}\vector{b}\|_{2}}$
of MINRES-QLP with that of AB-RRGMRES using $B = A^{\rm T}$ $(C = I)$. 
The method for applying MINRES-QLP to the GP, and index2 inconsistent problems 
was explained in section \ref{MINQLP}.

Fig. \ref{qlp_gp} 
shows $\displaystyle \frac{\|A^{{\rm T}}\vector{r}_{k}\|_{2}}{\|A^{{\rm T}}\vector{b}\|_{2}}$ versus 
the number of iterations $k$ for AB-RRGMRES using $B=A^{\rm T}$ and 
MINRES-QLP applied to the same GP inconsistent problem as in section \ref{NumGP}.

Fig. \ref{qlp_id2} 
shows $\displaystyle \frac{\|A^{{\rm T}}\vector{r}_{k}\|_{2}}{\|A^{{\rm T}}\vector{b}\|_{2}}$ versus 
the number of iterations $k$ for AB-RRGMRES using $B=A^{\rm T}$ and
MINRES-QLP applied to the same index 2 inconsistent problem as in section \ref{NumId2}.

\begin{figure}[htbp]
\begin{center}
\includegraphics[scale=0.55]{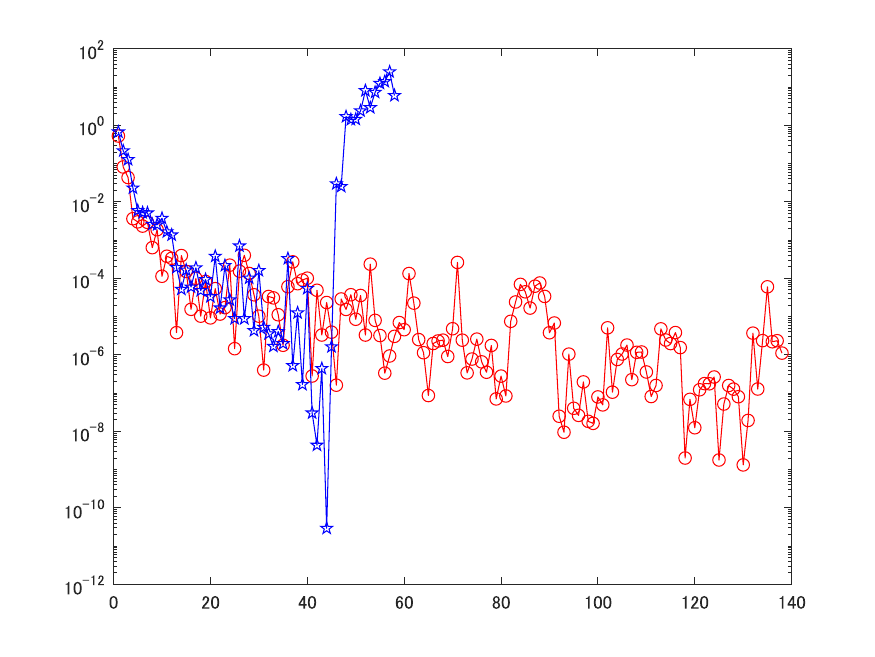}
\caption{$\displaystyle \frac{\|A^{{\rm T}}\vector{r}_{k}\|_{2}}{\|A^{{\rm T}}\vector{b}\|_{2}}$ versus 
	the number of iterations for AB-RRGMRES using $B=A^{\rm T}$ (blue, $\star$) and 
MINRES-QLP (red, $\circ$) applied to the GP inconsistent problem }
\label{qlp_gp}
\end{center}
\end{figure}

\begin{figure}[htbp]
\begin{center}
\includegraphics[scale=0.55]{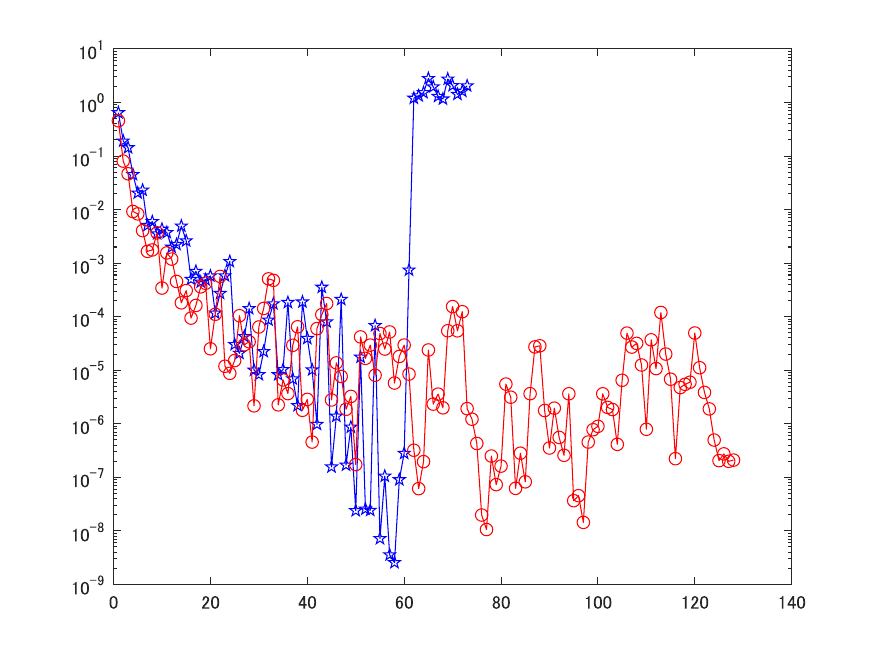}
\caption{$\displaystyle \frac{\|A^{{\rm T}}\vector{r}_{k}\|_{2}}{\|A^{{\rm T}}\vector{b}\|_{2}}$ versus 
the number of iterations for AB-RRGMRES using $B=A^{\rm T}$ (blue, $\star$) and 
MINRES-QLP (red, $\circ$) applied to the index 2 inconsistent problem}
\label{qlp_id2}
\end{center}
\end{figure}

We observe the following from Fig. \ref{qlp_gp} for the GP inconsistent problem.

\begin{itemize}
\item The minimum value of $\displaystyle \frac{\|A^{{\rm T}}\vector{r}_{k}\|_{2}}{\|A^{{\rm T}}\vector{b}\|_{2}}$
for AB-RRGMRES with $B = A^{\rm T}$ is less than $10^{-10}$, 
which is almost $5.5 \times 10^{-3}$ times that of MINRES-QLP,
and the minimum is reached at 1/3 the number of iterations of the latter.

\end{itemize}

We observe the following from Fig. \ref{qlp_id2} for the index 2 inconsistent
problem.

\begin{itemize}
\item The minimum value of $\displaystyle \frac{\|A^{{\rm T}}\vector{r}_{k}\|_{2}}{\|A^{{\rm T}}\vector{b}\|_{2}}$
for AB-RRGMRES with $B = A^{\rm T}$ is less than $10^{-8}$, 
which is almost $1.2 \times 10^{-2}$ times that of MINRES-QLP,
and the minimum is reached about $70\%$ the number of iterations of the latter.

\end{itemize}

\subsection{Underdetermined least squares problems}\label{NumLS}

For the least squares problem (\ref{lstsq2}) where
$A \in \rmn$ may be over- or underdetermined and is not necessarily of full rank, and
$\bb \in \Rm$ is not necessarily in $\ran(A)$,
we can also use the AB-RRGMRES method which applies RRGMRES to (\ref{AB})
where $ B \in \rnm $ satisfies $ \ran (AB) = \ran (A) $.

Considering the computational cost and required memory,
in this paper,
we apply AB-RRGMRES to the underdetermined least squares problem in \cite{HYI}.

Table \ref{matinfo} gives information on the test matrix from \cite{FOS}, including 
the number of rows $m$, the number of columns $n$, and the number of elements $nnz$.
The ${\rm T}$ in the name of the matrix denotes that the matrix was transposed.
Table \ref{matinfo} shows the effective size of the matrix after removing all zero columns
and zero rows.
In Table \ref{matinfo}, the definition of Rank and $\kappa(A)$ is the same as in Table \ref{gpinfo}.
Here, the information of matrices except ${\rm Maragal\_5T}$ in Table \ref{matinfo} is based on 
\cite[Table 1]{MH}. The Rank and $\kappa(A)$ of ${\rm Maragal\_8T}$ could not be computed on 
the computer of the author of \cite{MH} due to insufficient memory.

\begin{table}[htbp]
\centering
\caption{Characteristics of the coefficient matrices of the underdetermined least squares problems }
\label{matinfo}
\begin{tabular}{|c|r|r|r|r|r|}
\hline
Matrix & m & n & nnz  & Rank & $\kappa(A)$ \\
\hline
\hline
${\rm Maragal\_5T}$ &  3,296 & 4,654 & 93,091 & 2,147 &  {\rm $1.19 \times 10^5$} \\
\hline
${\rm lp\_dfl001}$ &  6,071 & 12,230 & 35,632 & 6,058 &  {\rm $3.49 \times 10^2$} \\
\hline
${\rm landmarkT}$ &  2,673 & 71,952 & 1,146,848 & 2,671 &  {\rm $1.02 \times 10^8$} \\
\hline
${\rm Maragal\_6T}$ &  10,144 & 21,251 & 537,694 & 8,331 &  {\rm $2.91 \times 10^6$} \\
\hline
${\rm Maragal\_7T}$ &  26,525 & 46,845 & 1,200,537 & 20,843 &  {\rm $8.98 \times 10^6$} \\
\hline
${\rm Maragal\_8}$ &  33,093 & 60,845 & 1,308,415 &  &   \\
\hline
${\rm lp\_cre\_a}$ &  3,428 & 7,248 & 18,168 & 3,423 &  {\rm $2.11 \times 10^4$} \\
\hline
\end{tabular}
\end{table}

The right hand side vector $\vector{b}$ was set as 
$\vector{b} = \frac{A \times (1,1,.,1)^{{\rm T}}}{\|A \times (1,1,...,1)^{{\rm T}}\|_{2}}
 + \frac{\vector{u}(0,1)}{\|\vector{u}(0,1)\|_{2}}\times 0.01$
, where $\vector{u}(0,1)$ 
is an $m$ dimensional vector of pseudorandom numbers generated according to the uniform 
distribution in the interval [0,1].
Hence, generically $\vector{b} \notin \ran(A)$.

For AB-RRGMRES using NR-SSOR for ${\rm Maragal\_5T}$,
we determined the number of inner iterations $\ell$
which is the smallest 
among 1,2,4,6,8,10 inner iterations such that
the minimum value of
$\displaystyle 
\frac{\|A^{{\rm T}}\vector{r}_{k}\|_{2}}{\|A^{{\rm T}}\vector{b}\|_{2}}$ 
becomes smaller than $5.0 \times 10^{-10}$
when the relaxation parameter $\omega $ is $1.0$.
Next, we determined that the number of inner iterations $\ell$ of AB-GMRES using NR-SSOR
is similarly.
Furthermore, we determined the optimal $\omega$ which minimizes
$\displaystyle 
\frac{\|A^{{\rm T}}\vector{r}_{k}\|_{2}}{\|A^{{\rm T}}\vector{b}\|_{2}}$ with $\ell$ inner iterations 
by searching among $\omega = 1.8,1.6,1.4,1.2,1.1,1.0,0.9,0.8,0.6,0.4,0.2$.
As a result, for ${\rm Maragal\_5T}$, we chose $\omega = 1.1$ and $4$ as the number of inner iterations for
AB-RRGMRES and $\omega = 0.8$ and $4$ as the number of inner iterations for
AB-GMRES. 

For AB-RRGMRES using NR-SSOR for ${\rm lp\_dfl001}$,
we determined the number of inner iterations $\ell$
which is the smallest among 1,2,4,6,8,10 inner iterations such that
the minimum value of
$\displaystyle 
\frac{\|A^{{\rm T}}\vector{r}_{k}\|_{2}}{\|A^{{\rm T}}\vector{b}\|_{2}}$
of AB-RRGMRES using NR-SSOR is the smallest
when the relaxation parameter $\omega $ is $1.0$.
Next, we determined that the number of inner iterations $\ell$ of AB-GMRES using NR-SSOR
is similarly.
Then, we determined the optimal $\omega$ which minimizes
$\displaystyle 
\frac{\|A^{{\rm T}}\vector{r}_{k}\|_{2}}{\|A^{{\rm T}}\vector{b}\|_{2}}$
of AB-RRGMRES and AB-GMRES using NR-SSOR with $\ell$ inner iterations 
by searching among $\omega = 1.8,1.6,1.4,1.2,1.1,1.0,0.9,0.8,0.6,0.4,0.2$.
As a result, we chose $\omega = 1.4$ and $6$ as the number of inner iterations for
AB-RRGMRES and $\omega = 1.6$ and $6$ as the number of inner iterations for
AB-GMRES.  

For the computation time for AB-RRGMRES for the least squares problems ${\rm Maragal\_5T}$
and ${\rm lp\_dfl001}$, an average was taken over 10 measurements with the same vector $\vector{b}$.

Fig. \ref{acatund_i} for ${\rm Maragal\_5T}$ and Fig. \ref{acatlp_i} for ${\rm lp\_dfl001}$ 
show $\displaystyle \frac{\|A^{{\rm T}}\vector{r}_{k}\|_{2}}{\|A^{{\rm T}}\vector{b}\|_{2}}$ versus 
the number of iterations $k$ for AB-RRGMRES using NR-SSOR, 
$B=\{{\rm diag}(A^{\rm T}A)\}^{-1} A$ 
and $B = A^{\rm T}$, respectively. 

Table \ref{rss_nop_cpu} for ${\rm Maragal\_5T}$ shows the number of iterations 
and the computation time in seconds for AB-RRGMRES using NR-SSOR with $4$ inner iterations and 
$\omega = 1.1$, $B=\{{\rm diag}(A^{\rm T}A)\}^{-1} A$
and using $B = A^{\rm T}$ when the convergence criterion is 
$\displaystyle \frac{\|A^{{\rm T}}\vector{r}_{k}\|_{2}}{\|A^{{\rm T}}\vector{b}\|_{2}} < 10^{-8}$.

Table \ref{lp_cpu} for ${\rm lp\_dfl001}$ shows the number of iterations 
and the computation time in seconds of AB-RRGMRES using NR-SSOR with $6$ inner iterations 
and $\omega = 1.4$, 
using $B=\{{\rm diag}(A^{\rm T}A)\}^{-1} A$
and using $B = A^{\rm T}$ when the convergence criterion is 
$\displaystyle \frac{\|A^{{\rm T}}\vector{r}_{k}\|_{2}}{\|A^{{\rm T}}\vector{b}\|_{2}} < 10^{-10}$.

\begin{figure}[htbp]
\begin{center}
\includegraphics[scale=0.55]{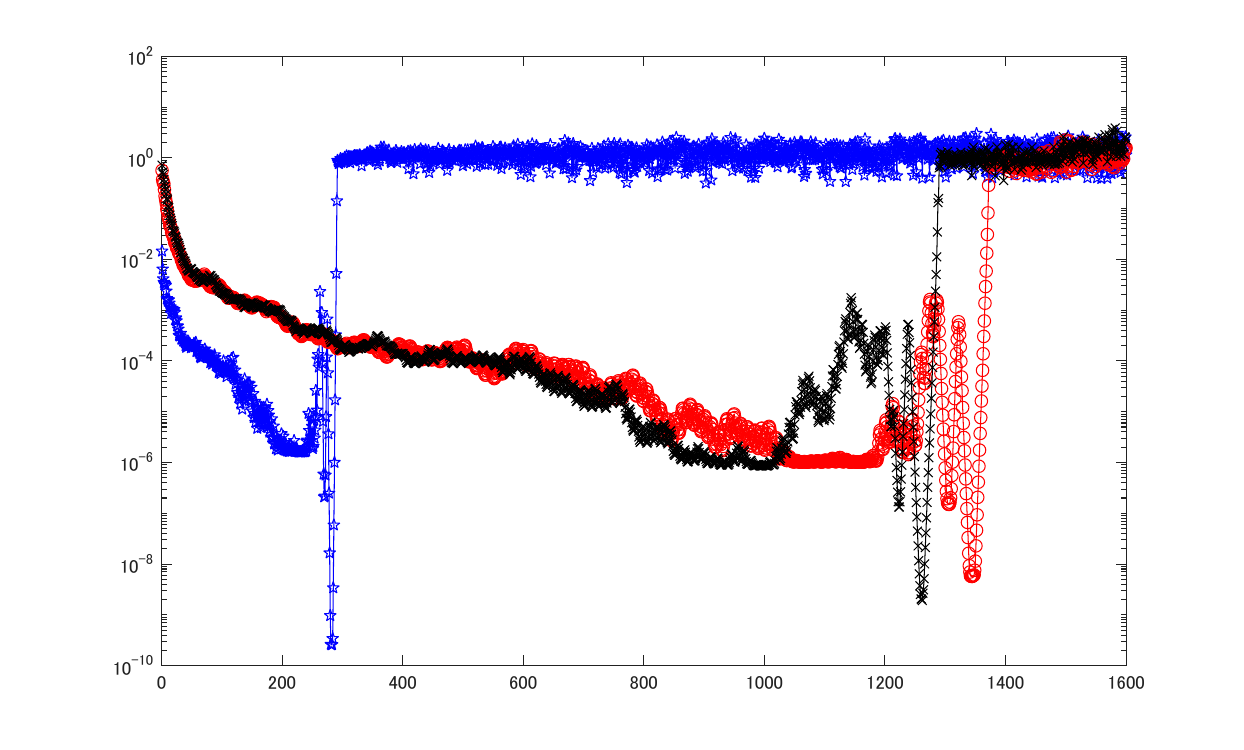}
\caption{$\displaystyle \frac{\|A^{{\rm T}}\vector{r}_{k}\|_{2}}{\|A^{{\rm T}}\vector{b}\|_{2}}$ versus 
the number of iterations of AB-RRGMRES using NR-SSOR
with 4 inner iterations and $\omega = 1.1$ (blue, $\star$), 
$B=\{{\rm diag}(A^{\rm T}A)\}^{-1} A$ (red, $\circ$), 
and $B=A^{\rm T}$ (black, $\times$) for the underdetermined least squares problem ${\rm Maragal\_5T}$}
\label{acatund_i}
\end{center}
\end{figure}

\begin{table}[htbp]
\centering
\caption{Computation results for the underdetermined least squares problem (${\rm Maragal\_5T}$) 
Iter: number of iterations, Tno: computation time not including the computation time for
computing the relative residual norm 
(Convergence criterion : $\displaystyle \frac{\|A^{{\rm T}}\vector{r}_{k}\|_{2}}{\|A^{{\rm T}}\vector{b}\|_{2}} < 10^{-8}$)}
\label{rss_nop_cpu}
\begin{tabular}{|c|r|r|r|r|r|r|}
\hline
Method & Iter & Tno [sec] \\
\hline
\hline
AB-RRGMRES using NR-SSOR with 4 inner iterations and $\omega = 1.1$ & 280  & 1.66 \\
\hline
AB-RRGMRES using $B=\{{\rm diag}(A^{\rm T}A)\}^{-1} A$  & 1,340 & 16.1 \\
\hline
AB-RRGMRES using $B = A^{\rm T}$ & 1,257 & 14.2 \\
\hline
\end{tabular}
\end{table}

\begin{figure}[htbp]
\begin{center}
\includegraphics[scale=0.55]{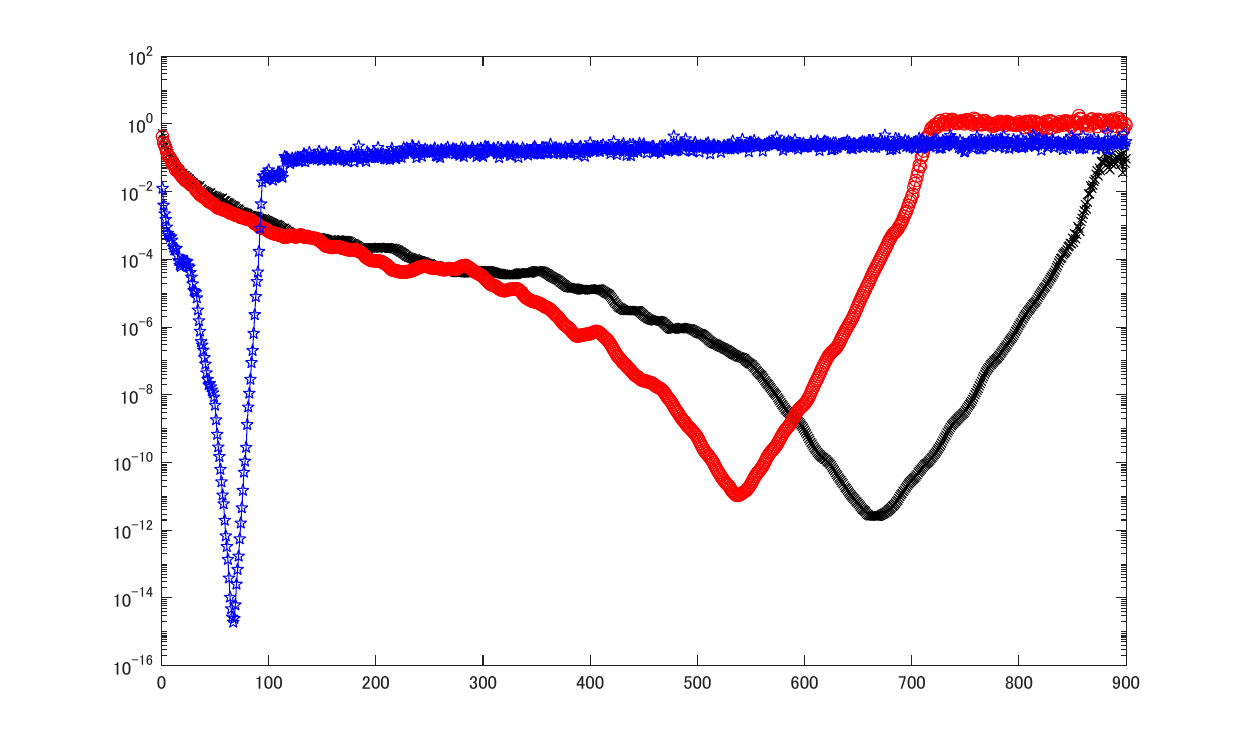}
\caption{$\displaystyle \frac{\|A^{{\rm T}}\vector{r}_{k}\|_{2}}{\|A^{{\rm T}}\vector{b}\|_{2}}$ versus 
the number of iterations of AB-RRGMRES using NR-SSOR
with 6 inner iterations and $\omega = 1.4$ (blue, $\star$), 
$B=\{{\rm diag}(A^{\rm T}A)\}^{-1} A$ (red, $\circ$), 
and $B=A^{\rm T}$ (black, $\times$) for the underdetermined least squares problem ${\rm lp\_dfl001}$}
\label{acatlp_i}
\end{center}
\end{figure}

\begin{table}[htbp]
\centering
\caption{Computation results for the underdetermined least squares problem (${\rm lp\_dfl001}$) 
Iter: number of iterations, Tno: computation time not including the computation time for
computing the relative residual norm 
(Convergence criterion : $\displaystyle \frac{\|A^{{\rm T}}\vector{r}_{k}\|_{2}}{\|A^{{\rm T}}\vector{b}\|_{2}} < 10^{-10}$)}
\label{lp_cpu}
\begin{tabular}{|c|r|r|r|r|r|r|}
\hline
Method & Iter & Tno [sec] \\
\hline
\hline
AB-RRGMRES using NR-SSOR with 6 inner iterations and $\omega = 1.4$ & 55   & 0.27  \\
\hline
AB-RRGMRES using $B=\{{\rm diag}(A^{\rm T}A)\}^{-1} A$  & 515   & 3.73 \\
\hline
AB-RRGMRES using $B = A^{\rm T}$ & 623 & 5.49  \\
\hline
\end{tabular}
\end{table}

Fig. \ref{abgmrrund_i} for ${\rm Maragal\_5T}$ and Fig. \ref{abgmrrlp_i} for ${\rm lp\_dfl001}$ 
show $\displaystyle \frac{\|A^{{\rm T}}\vector{r}_{k}\|_{2}}{\|A^{{\rm T}}\vector{b}\|_{2}}$ versus 
the number of iterations $k$ for AB-RRGMRES using NR-SSOR
and AB-GMRES using NR-SSOR.

Table \ref{min_m5t} for ${\rm Maragal\_5T}$ and Table \ref{min_dfl} for ${\rm lp\_dfl001}$
show the minimum value of $\displaystyle \frac{\|A^{{\rm T}}\vector{r}_{k}\|_{2}}{\|A^{{\rm T}}\vector{b}\|_{2}}$ 
and the number of iterations for attaining the minimum value for AB-RRGMRES using NR-SSOR
and AB-GMRES using NR-SSOR.

\begin{figure}[htbp]
\begin{center}
\includegraphics[scale=0.55]{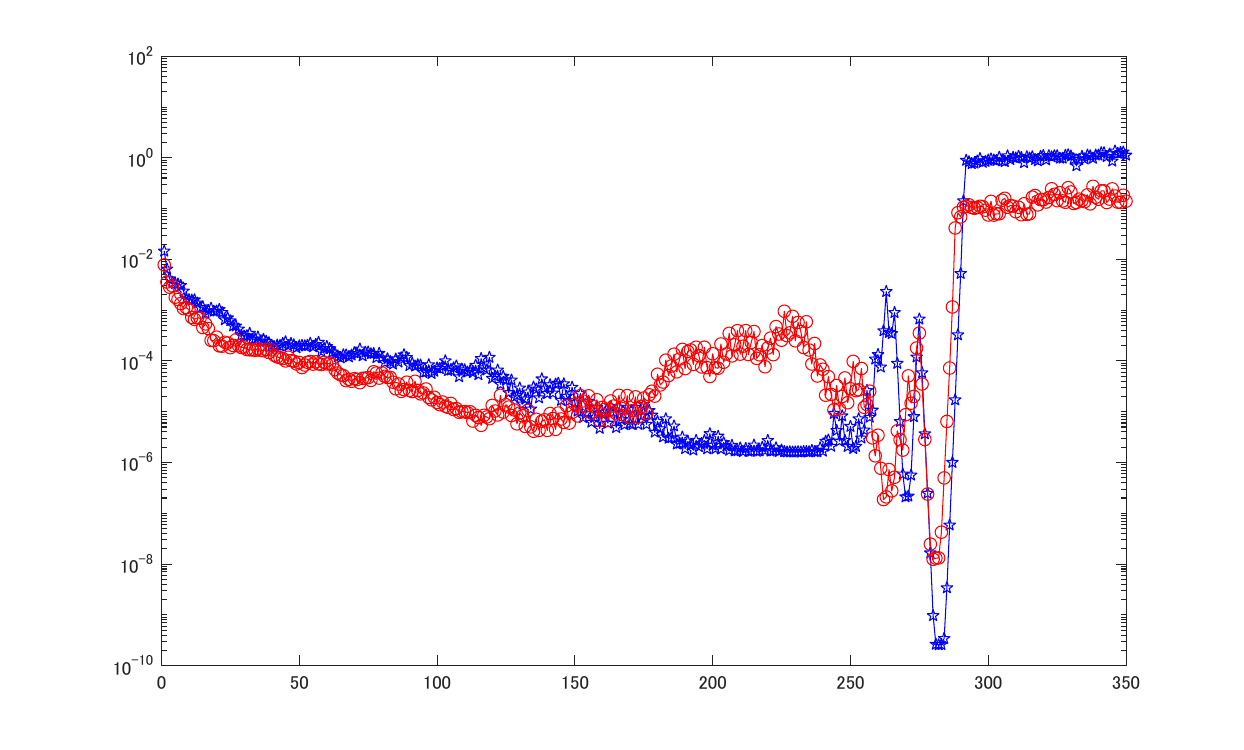}
\caption{$\displaystyle \frac{\|A^{{\rm T}}\vector{r}_{k}\|_{2}}{\|A^{{\rm T}}\vector{b}\|_{2}}$ versus 
the number of iterations for AB-RRGMRES using NR-SSOR
with 4 inner iteration and $\omega =1.1$ (blue, $\star$) and 
AB-GMRES using NR-SSOR
with 4 inner iteration and $\omega =0.8$ (red, $\circ$) 
	for the underdetermined least squares problem ${\rm Maragal\_5T}$}
\label{abgmrrund_i}
\end{center}
\end{figure}

\begin{figure}[htbp]
\begin{center}
\includegraphics[scale=0.55]{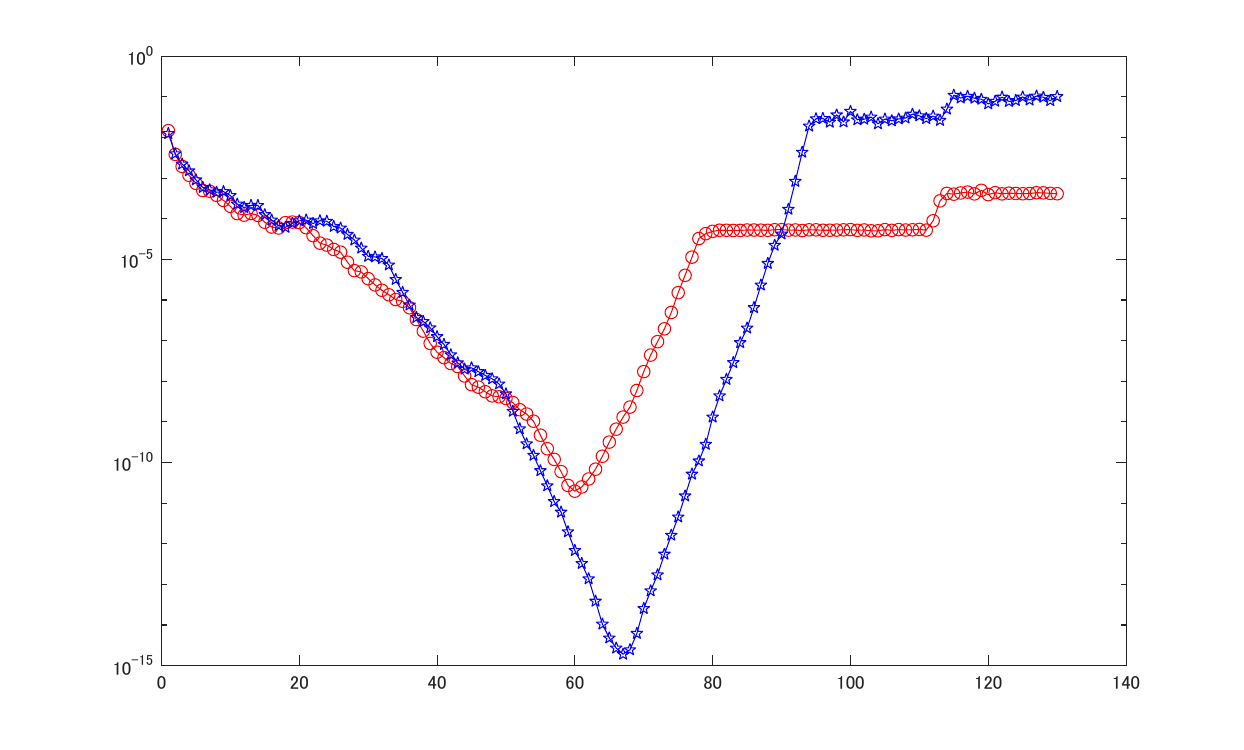}
\caption{$\displaystyle \frac{\|A^{{\rm T}}\vector{r}_{k}\|_{2}}{\|A^{{\rm T}}\vector{b}\|_{2}}$ versus 
the number of iterations for AB-RRGMRES using NR-SSOR
with 6 inner iteration and $\omega =1.4$ (blue, $\star$) and 
AB-GMRES using NR-SSOR
with 6 inner iteration and $\omega =1.6$ (red, $\circ$) 
	for the underdetermined least squares problem ${\rm lp\_dfl001}$}
\label{abgmrrlp_i}
\end{center}
\end{figure}

\begin{table}[htbp]
\centering
\caption{Computation results for the underdetermined least squares problem (${\rm Maragal\_5T}$) 
Iter: number of iterations, $\displaystyle \min \frac{\|A^{{\rm T}}\vector{r}_{k}\|_{2}}{\|A^{{\rm T}}\vector{b}\|_{2}}$ : 
minimum value of $\displaystyle \frac{\|A^{{\rm T}}\vector{r}_{k}\|_{2}}{\|A^{{\rm T}}\vector{b}\|_{2}}$}
\label{min_m5t}
\begin{tabular}{|c|r|r|}
\hline
Method & Iter &  $\displaystyle \min \frac{\|A^{{\rm T}}\vector{r}_{k}\|_{2}}{\|A^{{\rm T}}\vector{b}\|_{2}}$ \\
\hline
\hline
AB-RRGMRES using NR-SSOR with 4 inner iterations and $\omega = 1.1$ & 282   & $2.56 \times 10^{-10}$ \\
\hline
AB-GMRES using NR-SSOR with 4 inner iterations and $\omega = 0.8$  & 280   &  $1.24 \times 10^{-8}$ \\
\hline
\end{tabular}
\end{table}

\begin{table}[htbp]
\centering
\caption{Computation results for the underdetermined least squares problem (${\rm lp\_dfl001}$) 
Iter: number of iterations, $\displaystyle \min \frac{\|A^{{\rm T}}\vector{r}_{k}\|_{2}}{\|A^{{\rm T}}\vector{b}\|_{2}}$ : 
minimum value of $\displaystyle \frac{\|A^{{\rm T}}\vector{r}_{k}\|_{2}}{\|A^{{\rm T}}\vector{b}\|_{2}}$}
\label{min_dfl}
\begin{tabular}{|c|r|r|}
\hline
Method & Iter &  $\min \displaystyle \frac{\|A^{{\rm T}}\vector{r}_{k}\|_{2}}{\|A^{{\rm T}}\vector{b}\|_{2}}$ \\
\hline
\hline
AB-RRGMRES using NR-SSOR with 6 inner iterations and $\omega = 1.4$ & 67   & $1.91 \times 10^{-15}$ \\
\hline
AB-GMRES using NR-SSOR with 6 inner iterations and $\omega = 1.6$  & 60   &  $1.97 \times 10^{-11}$ \\
\hline
\end{tabular}
\end{table}

We observe the following from Fig. \ref{acatund_i}, \ref{abgmrrund_i}
, Table \ref{rss_nop_cpu} and Table \ref{min_m5t} for ${\rm Maragal\_5T}$.

\begin{itemize}

\item The minimum value of $\frac{\|A^{{\rm T}}\vector{r}_{k}\|_{2}}{\|A^{{\rm T}}\vector{b}\|_{2}}$ 
for AB-RRGMRES 
using NR-SSOR with 4 inner iterations and $\omega = 1.1$ 
is less than $10^{-9}$, and almost $1/10$ of $B=A^{\rm T}$. 
On the other hand, the minimum value of 
$\frac{\|A^{{\rm T}}\vector{r}_{k}\|_{2}}{\|A^{{\rm T}}\vector{b}\|_{2}}$ for AB-RRGMRES 
with $B=\{{\rm diag}(A^{\rm T}A)\}^{-1} A^{\rm T}$ is a little larger than
with $B = A^{\rm T}$.

\item When the convergence criterion is 
$\displaystyle \frac{\|A^{{\rm T}}\vector{r}_{k}\|_{2}}{\|A^{{\rm T}}\vector{b}\|_{2}} < 10^{-8}$,
AB-RRGMRES using NR-SSOR with 4 inner iterations and $\omega = 1.1$ 
is $8.55$, $9.70$ times faster 
in the computation time compared to
AB-RRGMRES using $B = A^{\rm T}$, $B=\{{\rm diag}(A^{\rm T}A)\}^{-1} A^{\rm T}$.
Furthermore, the ratio of reduction of the computation time is higher
than that of number of iterations because AB-RRGMRES solves
the least squares problem with Hessenberg matrix whose size is equal to number of iterations
and the order of the computational cost of the Arnoldi process is higher
than the number of iterations.

\item The minimum value of 
$\frac{\|A^{\rm T}\vector{r}_{k}\|_{2}}{\|A^{\rm T}\vector{b}\|_{2}}$ for 
AB-RRGMRES 
using NR-SSOR with 4 inner iteration and $\omega = 1.1$ is $2.06 \times 10^{-2}$ times 
that of AB-GMRES using NR-SSOR with 4 inner iteration and $\omega = 0.8$.

\end{itemize}

We observe the following from Fig. \ref{acatlp_i}, \ref{abgmrrlp_i}
, Table \ref{lp_cpu} and Table \ref{min_dfl} for ${\rm lp\_dfl001}$.

\begin{itemize}

\item The minimum value of $\frac{\|A^{{\rm T}}\vector{r}_{k}\|_{2}}{\|A^{{\rm T}}\vector{b}\|_{2}}$ 
for AB-RRGMRES 
using NR-SSOR with 6 inner iterations and $\omega = 1.4$ 
is less than $10^{-14}$, and almost $10^{-3}$ of $B=A^{\rm T}$
and $B=\{{\rm diag}(A^{\rm T}A)\}^{-1} A^{\rm T}$. 
On the other hand, the minimum value of 
$\frac{\|A^{{\rm T}}\vector{r}_{k}\|_{2}}{\|A^{{\rm T}}\vector{b}\|_{2}}$ for AB-RRGMRES 
with $B = A^{\rm T}$ is a little smaller than 
with $B=\{{\rm diag}(A^{\rm T}A)\}^{-1} A^{\rm T}$.

\item When the convergence criterion is 
$\displaystyle \frac{\|A^{{\rm T}}\vector{r}_{k}\|_{2}}{\|A^{{\rm T}}\vector{b}\|_{2}} < 10^{-10}$,
AB-RRGMRES using NR-SSOR with 6 inner iterations and $\omega = 1.4$ 
is $20.3$ times faster 
in computation time compared to
AB-RRGMRES using $B = A^{\rm T}$ and $13.8$ of $B=\{{\rm diag}(A^{\rm T}A)\}^{-1} A^{\rm T}$.
Furthermore, the ratio of reduction of the computation time is higher
than that of number of iterations because AB-RRGMRES solves
the least squares problem with Hessenberg matrix whose size is equal to number of iterations
and the order of the computational cost of the Arnoldi process is higher
than the number of iterations.

\item The minimum value of 
$\frac{\|A^{\rm T}\vector{r}_{k}\|_{2}}{\|A^{\rm T}\vector{b}\|_{2}}$ for 
AB-RRGMRES 
using NR-SSOR with 6 inner iteration and $\omega = 1.4$ is $0.97 \times 10^{-4}$ times 
that of AB-GMRES using NR-SSOR with 6 inner iteration and $\omega = 1.6$.

\end{itemize}

Next, we compare the number of iterations and the computation time 
for AB-RRGMRES using NR-SSOR, $B=\{{\rm diag}(A^{\rm T}A)\}^{-1} A$
and $B = A^{\rm T}$
for the underdetermined least squares problems ${\rm landmarkT}$, ${\rm Maragal\_6T}$, 
${\rm Maragal\_7T}$, ${\rm Maragal\_8}$ and ${\rm lp\_cre\_a}$ 
in Table \ref{matinfo}.
Furthermore, we compare the minimum value of $\frac{\|A^{{\rm T}}\vector{r}_{k}\|_{2}}{\|A^{{\rm T}}\vector{b}\|_{2}}$
and the number of iterations at the minimum value
of AB-RRGMRES using NR-SSOR with that of AB-GMRES using NR-SSOR 
for the underdetermined least squares problems ${\rm Maragal\_5T}$, ${\rm lp\_dfl001}$, 
${\rm landmarkT}$, ${\rm Maragal\_6T}$, 
${\rm Maragal\_7T}$, ${\rm Maragal\_8}$ and ${\rm lp\_cre\_a}$ 
in Table \ref{matinfo}.

For AB-RRGMRES using NR-SSOR for ${\rm landmarkT}$, ${\rm Maragal\_6T}$, 
${\rm Maragal\_7T}$, ${\rm Maragal\_8}$ and ${\rm lp\_cre\_a}$,
we determined the number of inner iterations $\ell$
which is the inner iteration 
among 1,2,4,6,8,10 inner iterations such that
the minimum value of
$\displaystyle 
\frac{\|A^{{\rm T}}\vector{r}_{k}\|_{2}}{\|A^{{\rm T}}\vector{b}\|_{2}}$
of AB-RRGMRES using NR-SSOR is the smallest
when the relaxation parameter $\omega $ is $1.0$.
Next, we determined the number of inner iterations $\ell$ of AB-GMRES using NR-SSOR
similarly to that of AB-RRGMRES using NR-SSOR.
Furthermore, we determined the optimal $\omega$ which minimizes
$\displaystyle 
\frac{\|A^{{\rm T}}\vector{r}_{k}\|_{2}}{\|A^{{\rm T}}\vector{b}\|_{2}}$
of AB-RRGMRES and AB-GMRES using NR-SSOR with $\ell$ inner iterations 
by searching among $\omega = 1.8,1.6,1.4,1.2,1.1,1.0,0.9,0.8,0.6,0.4,0.2$.
Then, for the matrices in Table \ref{matinfo}, 
the optimal value for 
$\omega$ and inner iterations $\ell$ for AB-RRGMRES using NR-SSOR
are described in Table \ref{omegaL}
and those for AB-GMRES using NR-SSOR are described in 
Table \ref{omegaL_gm}.

\begin{table}[htbp]
\centering
\caption{Optimal $\omega$ and inner iterations $\ell$ for AB-RRGMRES using NR-SSOR}
\label{omegaL}
\begin{tabular}{|c|r|r|}
\hline
Matrix & $\omega$ & $\ell$  \\
\hline
\hline
${\rm Maragal\_5T}$ &  1.1 & 4 \\
\hline
${\rm lp\_dfl001}$ &  1.4 & 6 \\
\hline
${\rm landmarkT}$ &  0.2 & 2  \\
\hline
${\rm Maragal\_6T}$ &  1.0 & 10  \\
\hline
${\rm Maragal\_7T}$ &  0.4 & 2  \\
\hline
${\rm Maragal\_8}$ &  1.0 & 4    \\
\hline
${\rm lp\_cre\_a}$ &  1.0 & 8 \\
\hline
\end{tabular}
\end{table}

\begin{table}[htbp]
\centering
\caption{Optimal $\omega$ and inner iterations $\ell$ for AB-GMRES using NR-SSOR}
\label{omegaL_gm}
\begin{tabular}{|c|r|r|}
\hline
Matrix & $\omega$ & $\ell$  \\
\hline
\hline
${\rm Maragal\_5T}$ & 0.8  & 4 \\
\hline
${\rm lp\_dfl001}$ & 1.6  & 6 \\
\hline
${\rm landmarkT}$ &  0.2 & 2  \\
\hline
${\rm Maragal\_6T}$ &  1.1 & 10  \\
\hline
${\rm Maragal\_7T}$ &  0.4 & 2  \\
\hline
${\rm Maragal\_8}$ &  0.9 & 4    \\
\hline
${\rm lp\_cre\_a}$ &  0.2 & 8 \\
\hline
\end{tabular}
\end{table}

For the computation time for AB-RRGMRES using NR-SSOR for the least squares problems ${\rm landmarkT}$, 
${\rm Maragal\_6T}$, ${\rm Maragal\_7T}$, ${\rm Maragal\_8}$ and ${\rm lp\_cre\_a}$, 
an average was taken over 10 measurements with the same vector $\vector{b}$.
Furthermore, for the computation time for AB-RRGMRES using $B=\{{\rm diag}(A^{\rm T}A)\}^{-1} A^{\rm T}$
and $B = A^{\rm T}$
for the least squares problems ${\rm Maragal\_6T}$, ${\rm Maragal\_7T}$ and
${\rm Maragal\_8}$, only one measurement was taken since
the computation time was very large due to large $\min$\{m,n\} of matrices in Table \ref{matinfo} 
and large number of iterations were required for convergence.
On the other hand, for the computation time for AB-RRGMRES using $B=\{{\rm diag}(A^{\rm T}A)\}^{-1} A^{\rm T}$
and $B = A^{\rm T}$
for the least squares problems ${\rm landmarkT}$ and ${\rm lp\_cre\_a}$,
an average was taken over 10 measurements with the same vector $\vector{b}$
since the computation time was not so large due to the size of $\min$\{m,n\} of the matrices
in Table \ref{matinfo}.

Table \ref{landT_cpu} for ${\rm landmarkT}$ shows the number of iterations 
and the computation time in seconds of AB-RRGMRES using NR-SSOR with $2$ inner iterations 
and $\omega = 0.2$, 
using $B=\{{\rm diag}(A^{\rm T}A)\}^{-1} A$
and using $B = A^{\rm T}$ when the convergence criterion is 
$\displaystyle \frac{\|A^{{\rm T}}\vector{r}_{k}\|_{2}}{\|A^{{\rm T}}\vector{b}\|_{2}} < 10^{-8}$.
Table \ref{m6T_cpu} for ${\rm Maragal\_6T}$ shows the number of iterations 
and the computation time in seconds of AB-RRGMRES using NR-SSOR with $10$ inner iterations 
and $\omega = 1.0$, 
using $B=\{{\rm diag}(A^{\rm T}A)\}^{-1} A$
and using $B = A^{\rm T}$ when the convergence criterion is 
$\displaystyle \frac{\|A^{{\rm T}}\vector{r}_{k}\|_{2}}{\|A^{{\rm T}}\vector{b}\|_{2}} < 10^{-6}$.
Table \ref{m7T_cpu} for ${\rm Maragal\_7T}$ shows the number of iterations 
and the computation time in seconds of AB-RRGMRES using NR-SSOR with $2$ inner iterations 
and $\omega = 0.4$, 
using $B=\{{\rm diag}(A^{\rm T}A)\}^{-1} A$
and using $B = A^{\rm T}$ when the convergence criterion is 
$\displaystyle \frac{\|A^{{\rm T}}\vector{r}_{k}\|_{2}}{\|A^{{\rm T}}\vector{b}\|_{2}} < 10^{-7}$.
Table \ref{m8_cpu} for ${\rm Maragal\_8}$ shows the number of iterations 
and the computation time in seconds of AB-RRGMRES using NR-SSOR with $4$ inner iterations 
and $\omega = 1.0$, 
using $B=\{{\rm diag}(A^{\rm T}A)\}^{-1} A$
and using $B = A^{\rm T}$ when the convergence criterion is 
$\displaystyle \frac{\|A^{{\rm T}}\vector{r}_{k}\|_{2}}{\|A^{{\rm T}}\vector{b}\|_{2}} < 10^{-7}$.
Table \ref{cre_a_cpu} for ${\rm lp\_cre\_a}$ shows the number of iterations 
and the computation time in seconds of AB-RRGMRES using NR-SSOR with $8$ inner iterations 
and $\omega = 1.0$, 
using $B=\{{\rm diag}(A^{\rm T}A)\}^{-1} A$
and using $B = A^{\rm T}$ when the convergence criterion is 
$\displaystyle \frac{\|A^{{\rm T}}\vector{r}_{k}\|_{2}}{\|A^{{\rm T}}\vector{b}\|_{2}} < 10^{-10}$.
 
Table \ref{min_landt} for ${\rm landmarkT}$, Table \ref{min_m6t} for ${\rm Maragal\_6T}$,
Table \ref{min_m7t} for ${\rm Maragal\_7T}$, Table \ref{min_m8} for ${\rm Maragal\_8}$ 
and Table \ref{min_crea} for ${\rm lp\_cre\_a}$ 
show the minimum value of $\displaystyle \frac{\|A^{{\rm T}}\vector{r}_{k}\|_{2}}{\|A^{{\rm T}}\vector{b}\|_{2}}$ 
and the number of iterations for getting the minimum value for AB-RRGMRES using NR-SSOR
and AB-GMRES using NR-SSOR.

\begin{table}[htbp]
\centering
\caption{Computation results for the underdetermined least squares problem (${\rm landmarkT}$) 
Iter: number of iterations, Tno: computation time not including the computation time for
computing the relative residual norm 
(Convergence criterion : $\displaystyle \frac{\|A^{{\rm T}}\vector{r}_{k}\|_{2}}{\|A^{{\rm T}}\vector{b}\|_{2}} < 10^{-8}$)}
\label{landT_cpu}
\begin{tabular}{|c|r|r|r|r|r|r|}
\hline
Method & Iter & Tno [sec] \\
\hline
\hline
AB-RRGMRES using NR-SSOR with 2 inner iterations and $\omega = 0.2$ & 1,262   & 47.2  \\
\hline
AB-RRGMRES using $B=\{{\rm diag}(A^{\rm T}A)\}^{-1} A$  & 2,664   & 98.9 \\
\hline
AB-RRGMRES using $B = A^{\rm T}$ & 2,664 & 95.8  \\
\hline
\end{tabular}
\end{table}

\begin{table}[htbp]
\centering
\caption{Computation results for the underdetermined least squares problem (${\rm Maragal\_6T}$) 
Iter: number of iterations, Tno: computation time not including the computation time for
computing the relative residual norm 
(Convergence criterion : $\displaystyle \frac{\|A^{{\rm T}}\vector{r}_{k}\|_{2}}{\|A^{{\rm T}}\vector{b}\|_{2}} < 10^{-6}$)}
\label{m6T_cpu}
\begin{tabular}{|c|r|r|r|r|r|r|}
\hline
Method & Iter & Tno [sec] \\
\hline
\hline
AB-RRGMRES using NR-SSOR with 10 inner iterations and $\omega = 1.0$ & 356   & 22.3  \\
\hline
AB-RRGMRES using $B=\{{\rm diag}(A^{\rm T}A)\}^{-1} A$  & 2,805   & 234.6 \\
\hline
AB-RRGMRES using $B = A^{\rm T}$ & 2,618 & 180.1  \\
\hline
\end{tabular}
\end{table}

\begin{table}[htbp]
\centering
\caption{Computation results for the underdetermined least squares problem (${\rm Maragal\_7T}$) 
Iter: number of iterations, Tno: computation time not including the computation time for
computing the relative residual norm 
(Convergence criterion : $\displaystyle \frac{\|A^{{\rm T}}\vector{r}_{k}\|_{2}}{\|A^{{\rm T}}\vector{b}\|_{2}} < 10^{-7}$)}
\label{m7T_cpu}
\begin{tabular}{|c|r|r|r|r|r|r|}
\hline
Method & Iter & Tno [sec] \\
\hline
\hline
AB-RRGMRES using NR-SSOR with 2 inner iterations and $\omega = 0.4$ & 756   & 54.6  \\
\hline
AB-RRGMRES using $B=\{{\rm diag}(A^{\rm T}A)\}^{-1} A$  & 2,506   & 397.7 \\
\hline
AB-RRGMRES using $B = A^{\rm T}$ & 2,314 & 336.1  \\
\hline
\end{tabular}
\end{table}

\begin{table}[htbp]
\centering
\caption{Computation results for the underdetermined least squares problem (${\rm Maragal\_8}$) 
Iter: number of iterations, Tno: computation time not including the computation time for
computing the relative residual norm 
(Convergence criterion : $\displaystyle \frac{\|A^{{\rm T}}\vector{r}_{k}\|_{2}}{\|A^{{\rm T}}\vector{b}\|_{2}} < 10^{-7}$)}
\label{m8_cpu}
\begin{tabular}{|c|r|r|r|r|r|r|}
\hline
Method & Iter & Tno [sec] \\
\hline
\hline
AB-RRGMRES using NR-SSOR with 4 inner iterations and $\omega = 1.0$ & 421   & 39.1  \\
\hline
AB-RRGMRES using $B=\{{\rm diag}(A^{\rm T}A)\}^{-1} A$  & 2,941   & 689.5 \\
\hline
AB-RRGMRES using $B = A^{\rm T}$ & 6,445 & 3435.0  \\
\hline
\end{tabular}
\end{table}

\begin{table}[htbp]
\centering
\caption{Computation results for the underdetermined least squares problem (${\rm lp\_cre\_a}$) 
Iter: number of iterations, Tno: computation time not including the computation time for
computing the relative residual norm 
(Convergence criterion : $\displaystyle \frac{\|A^{{\rm T}}\vector{r}_{k}\|_{2}}{\|A^{{\rm T}}\vector{b}\|_{2}} < 10^{-10}$)}
\label{cre_a_cpu}
\begin{tabular}{|c|r|r|r|r|r|r|}
\hline
Method & Iter & Tno [sec] \\
\hline
\hline
AB-RRGMRES using NR-SSOR with 8 inner iterations and $\omega = 1.0$ & 183   & 0.658  \\
\hline
AB-RRGMRES using $B=\{{\rm diag}(A^{\rm T}A)\}^{-1} A$  & 1,191   & 12.9 \\
\hline
AB-RRGMRES using $B = A^{\rm T}$ & 2,260 & 52.6  \\
\hline
\end{tabular}
\end{table}

\begin{table}[htbp]
\centering
\caption{Computation results for the underdetermined least squares problem (${\rm landmarkT}$) 
Iter: number of iterations, $\displaystyle \min \frac{\|A^{{\rm T}}\vector{r}_{k}\|_{2}}{\|A^{{\rm T}}\vector{b}\|_{2}}$ : 
minimum value of $\displaystyle \frac{\|A^{{\rm T}}\vector{r}_{k}\|_{2}}{\|A^{{\rm T}}\vector{b}\|_{2}}$}
\label{min_landt}
\begin{tabular}{|c|r|r|}
\hline
Method & Iter &  $\min \displaystyle \frac{\|A^{{\rm T}}\vector{r}_{k}\|_{2}}{\|A^{{\rm T}}\vector{b}\|_{2}}$ \\
\hline
\hline
AB-RRGMRES using NR-SSOR with 2 inner iterations and $\omega = 0.2$ & 1,268   & $5.28 \times 10^{-9}$ \\
\hline
AB-GMRES using NR-SSOR with 2 inner iterations and $\omega = 0.2$  & 1,216   &  $2.90 \times 10^{-8}$ \\
\hline
\end{tabular}
\end{table}

\begin{table}[htbp]
\centering
\caption{Computation results for the underdetermined least squares problem (${\rm Maragal\_6T}$) 
Iter: number of iterations, $\displaystyle \min \frac{\|A^{{\rm T}}\vector{r}_{k}\|_{2}}{\|A^{{\rm T}}\vector{b}\|_{2}}$ : 
minimum value of $\displaystyle \frac{\|A^{{\rm T}}\vector{r}_{k}\|_{2}}{\|A^{{\rm T}}\vector{b}\|_{2}}$}
\label{min_m6t}
\begin{tabular}{|c|r|r|}
\hline
Method & Iter &  $\min \displaystyle \frac{\|A^{{\rm T}}\vector{r}_{k}\|_{2}}{\|A^{{\rm T}}\vector{b}\|_{2}}$ \\
\hline
\hline
AB-RRGMRES using NR-SSOR with 10 inner iterations and $\omega = 1.0$ & 486   & $5.67 \times 10^{-8}$ \\
\hline
AB-GMRES using NR-SSOR with 10 inner iterations and $\omega = 1.1$  & 380   &  $2.55 \times 10^{-7}$ \\
\hline
\end{tabular}
\end{table}

\begin{table}[htbp]
\centering
\caption{Computation results for the underdetermined least squares problem (${\rm Maragal\_7T}$) 
Iter: number of iterations, $\displaystyle \min \frac{\|A^{{\rm T}}\vector{r}_{k}\|_{2}}{\|A^{{\rm T}}\vector{b}\|_{2}}$ : 
minimum value of $\displaystyle \frac{\|A^{{\rm T}}\vector{r}_{k}\|_{2}}{\|A^{{\rm T}}\vector{b}\|_{2}}$}
\label{min_m7t}
\begin{tabular}{|c|r|r|}
\hline
Method & Iter &  $\min \displaystyle \frac{\|A^{{\rm T}}\vector{r}_{k}\|_{2}}{\|A^{{\rm T}}\vector{b}\|_{2}}$ \\
\hline
\hline
AB-RRGMRES using NR-SSOR with 2 inner iterations and $\omega = 0.4$ & 1,344   & $2.96 \times 10^{-8}$ \\
\hline
AB-GMRES using NR-SSOR with 2 inner iterations and $\omega = 0.4$  & 974   &  $1.97 \times 10^{-7}$ \\
\hline
\end{tabular}
\end{table}

\begin{table}[htbp]
\centering
\caption{Computation results for the underdetermined least squares problem (${\rm Maragal\_8}$) 
Iter: number of iterations, $\displaystyle \min \frac{\|A^{{\rm T}}\vector{r}_{k}\|_{2}}{\|A^{{\rm T}}\vector{b}\|_{2}}$ : 
minimum value of $\displaystyle \frac{\|A^{{\rm T}}\vector{r}_{k}\|_{2}}{\|A^{{\rm T}}\vector{b}\|_{2}}$}
\label{min_m8}
\begin{tabular}{|c|r|r|}
\hline
Method & Iter &  $\min \displaystyle \frac{\|A^{{\rm T}}\vector{r}_{k}\|_{2}}{\|A^{{\rm T}}\vector{b}\|_{2}}$ \\
\hline
\hline
AB-RRGMRES using NR-SSOR with 4 inner iterations and $\omega = 1.0$ & 445   & $4.83 \times 10^{-8}$ \\
\hline
AB-GMRES using NR-SSOR with 4 inner iterations and $\omega = 0.9$  & 426   &  $1.03 \times 10^{-7}$ \\
\hline
\end{tabular}
\end{table}

\begin{table}[htbp]
\centering
\caption{Computation results for the underdetermined least squares problem (${\rm lp\_cre\_a}$) 
Iter: number of iterations, $\displaystyle \min \frac{\|A^{{\rm T}}\vector{r}_{k}\|_{2}}{\|A^{{\rm T}}\vector{b}\|_{2}}$ : 
minimum value of $\displaystyle \frac{\|A^{{\rm T}}\vector{r}_{k}\|_{2}}{\|A^{{\rm T}}\vector{b}\|_{2}}$}
\label{min_crea}
\begin{tabular}{|c|r|r|}
\hline
Method & Iter &  $\min \displaystyle \frac{\|A^{{\rm T}}\vector{r}_{k}\|_{2}}{\|A^{{\rm T}}\vector{b}\|_{2}}$ \\
\hline
\hline
AB-RRGMRES using NR-SSOR with 8 inner iterations and $\omega = 1.0$ & 222   & $9.34 \times 10^{-15}$ \\
\hline
AB-GMRES using NR-SSOR with 8 inner iterations and $\omega = 0.2$  & 280   &  $8.40 \times 10^{-11}$ \\
\hline
\end{tabular}
\end{table}

We observe the following from Table \ref{landT_cpu} and Table \ref{min_landt} for ${\rm landmarkT}$.

\begin{itemize}

\item When the convergence criterion is 
$\displaystyle \frac{\|A^{{\rm T}}\vector{r}_{k}\|_{2}}{\|A^{{\rm T}}\vector{b}\|_{2}} < 10^{-8}$,
AB-RRGMRES using NR-SSOR with 2 inner iterations and $\omega = 0.2$ 
is $2.03$ times faster 
in computation time compared to
AB-RRGMRES using $B = A^{\rm T}$ and $2.10$ of $B=\{{\rm diag}(A^{\rm T}A)\}^{-1} A^{\rm T}$.
Furthermore, the ratio of reduction of the computation time is 
nearly equal to the ratio of reduction of the number of iterations for the convergence.

\item The minimum value of 
$\frac{\|A^{\rm T}\vector{r}_{k}\|_{2}}{\|A^{\rm T}\vector{b}\|_{2}}$ for 
AB-RRGMRES 
using NR-SSOR with 2 inner iteration and $\omega = 0.2$ is $1.82 \times 10^{-1}$ times 
that of AB-GMRES using NR-SSOR with 2 inner iteration and $\omega = 0.2$.

\end{itemize}

We observe the following from Table \ref{m6T_cpu} and Table \ref{min_m6t} for ${\rm Maragal\_6T}$.

\begin{itemize}

\item When the convergence criterion is 
$\displaystyle \frac{\|A^{{\rm T}}\vector{r}_{k}\|_{2}}{\|A^{{\rm T}}\vector{b}\|_{2}} < 10^{-6}$,
AB-RRGMRES using NR-SSOR with 10 inner iterations and $\omega = 1.0$ 
is $8.08$ times faster 
in computation time compared to
AB-RRGMRES using $B = A^{\rm T}$ and $10.5$ of $B=\{{\rm diag}(A^{\rm T}A)\}^{-1} A^{\rm T}$.
Furthermore, the ratio of reduction of the computation time is higher
than that of number of iterations because AB-RRGMRES solves
the least squares problem with Hessenberg matrix whose size is equal to number of iterations
and the order of the computational cost of the Arnoldi process is higher
than the number of iterations.

\item The minimum value of 
$\frac{\|A^{\rm T}\vector{r}_{k}\|_{2}}{\|A^{\rm T}\vector{b}\|_{2}}$ for 
AB-RRGMRES 
using NR-SSOR with 10 inner iteration and $\omega = 1.0$ is $2.22 \times 10^{-1}$ times 
that of AB-GMRES using NR-SSOR with 10 inner iteration and $\omega = 1.1$.

\end{itemize}

We observe the following from Table \ref{m7T_cpu} and Table \ref{min_m7t} for ${\rm Maragal\_7T}$.

\begin{itemize}

\item When the convergence criterion is 
$\displaystyle \frac{\|A^{{\rm T}}\vector{r}_{k}\|_{2}}{\|A^{{\rm T}}\vector{b}\|_{2}} < 10^{-7}$,
AB-RRGMRES using NR-SSOR with 10 inner iterations and $\omega = 1.0$ 
is $6.16$ times faster 
in computation time compared to
AB-RRGMRES using $B = A^{\rm T}$ and $7.28$ of $B=\{{\rm diag}(A^{\rm T}A)\}^{-1} A^{\rm T}$.
Furthermore, the ratio of reduction of the computation time is higher
than that of number of iterations because AB-RRGMRES solves
the least squares problem with Hessenberg matrix whose size is equal to number of iterations
and the order of the computational cost of the Arnoldi process is higher
than the number of iterations.

\item The minimum value of 
$\frac{\|A^{\rm T}\vector{r}_{k}\|_{2}}{\|A^{\rm T}\vector{b}\|_{2}}$ for 
AB-RRGMRES 
using NR-SSOR with 2 inner iteration and $\omega = 0.4$ is $1.50 \times 10^{-1}$ times 
that of AB-GMRES using NR-SSOR with 2 inner iteration and $\omega = 0.4$.

\end{itemize}

We observe the following from Table \ref{m7T_cpu} and Table \ref{min_m8} for ${\rm Maragal\_8}$.

\begin{itemize}

\item When the convergence criterion is 
$\displaystyle \frac{\|A^{{\rm T}}\vector{r}_{k}\|_{2}}{\|A^{{\rm T}}\vector{b}\|_{2}} < 10^{-7}$,
AB-RRGMRES using NR-SSOR with 4 inner iterations and $\omega = 1.0$ 
is $87.9$ times faster 
in computation time compared to
AB-RRGMRES using $B = A^{\rm T}$ and $17.6$ of $B=\{{\rm diag}(A^{\rm T}A)\}^{-1} A^{\rm T}$.
Furthermore, the ratio of reduction of the computation time is higher
than that of number of iterations because AB-RRGMRES solves
the least squares problem with Hessenberg matrix whose size is equal to number of iterations
and the order of the computational cost of the Arnoldi process is higher
than the number of iterations.

\item The minimum value of 
$\frac{\|A^{\rm T}\vector{r}_{k}\|_{2}}{\|A^{\rm T}\vector{b}\|_{2}}$ for 
AB-RRGMRES 
using NR-SSOR with 4 inner iteration and $\omega = 1.0$ is $4.69 \times 10^{-1}$ times 
that of AB-GMRES using NR-SSOR with 4 inner iteration and $\omega = 0.9$.

\end{itemize}

We observe the following from Table \ref{cre_a_cpu} and Table \ref{min_crea} for ${\rm lp\_cre\_a}$.

\begin{itemize}

\item When the convergence criterion is 
$\displaystyle \frac{\|A^{{\rm T}}\vector{r}_{k}\|_{2}}{\|A^{{\rm T}}\vector{b}\|_{2}} < 10^{-10}$,
AB-RRGMRES using NR-SSOR with 8 inner iterations and $\omega = 1.0$ 
is $79.9$ times faster 
in computation time compared to
AB-RRGMRES using $B = A^{\rm T}$ and $19.6$ of $B=\{{\rm diag}(A^{\rm T}A)\}^{-1} A^{\rm T}$.
Furthermore, the ratio of reduction of the computation time is higher
than that of number of iterations because AB-RRGMRES solves
the least squares problem with Hessenberg matrix whose size is equal to number of iterations
and the order of the computational cost of the Arnoldi process is higher
than the number of iterations.

\item The minimum value of 
$\frac{\|A^{\rm T}\vector{r}_{k}\|_{2}}{\|A^{\rm T}\vector{b}\|_{2}}$ for 
AB-RRGMRES 
using NR-SSOR with 8 inner iteration and $\omega = 1.0$ is $1.11 \times 10^{-4}$ times 
that of AB-GMRES using NR-SSOR with 8 inner iteration and $\omega = 0.2$.

\end{itemize}

\subsection{Comparison with MINRES-QLP for underdetermined least \\ squares problem}
For the underdetermined least squares problems
of ${\rm Maragal\_5T}$ and ${\rm lp\_dfl001}$
in section \ref{NumLS},
we compare $\displaystyle \frac{\|A^{{\rm T}}\vector{r}_{k}\|_{2}}{\|A^{{\rm T}}\vector{b}\|_{2}}$
of MINRES-QLP with that of AB-RRGMRES using $B = A^{\rm T}$ $(C = I)$. 
The method for applying MINRES-QLP to 
the underdetermined least squares problems ${\rm Maragal\_5T}$ and ${\rm lp\_dfl001}$
is explained in section \ref{MINQLP}.

Fig. \ref{qlp_m5}
shows $\displaystyle \frac{\|A^{{\rm T}}\vector{r}_{k}\|_{2}}{\|A^{{\rm T}}\vector{b}\|_{2}}$ versus 
the number of iterations $k$ for MINRES-QLP and AB-RRGMRES using $B = A^{\rm T}$ 
applied to the underdetermined least squares problem ${\rm Maragal\_5T}$.

Fig. \ref{qlp_lp}
shows $\displaystyle \frac{\|A^{{\rm T}}\vector{r}_{k}\|_{2}}{\|A^{{\rm T}}\vector{b}\|_{2}}$ versus 
the number of iterations $k$ for MINRES-QLP and AB-RRGMRES using $B = A^{\rm T}$ 
applied to the underdetermined least squares problem ${\rm lp\_dfl001}$.

\begin{figure}[htbp]
\begin{center}
\includegraphics[scale=0.55]{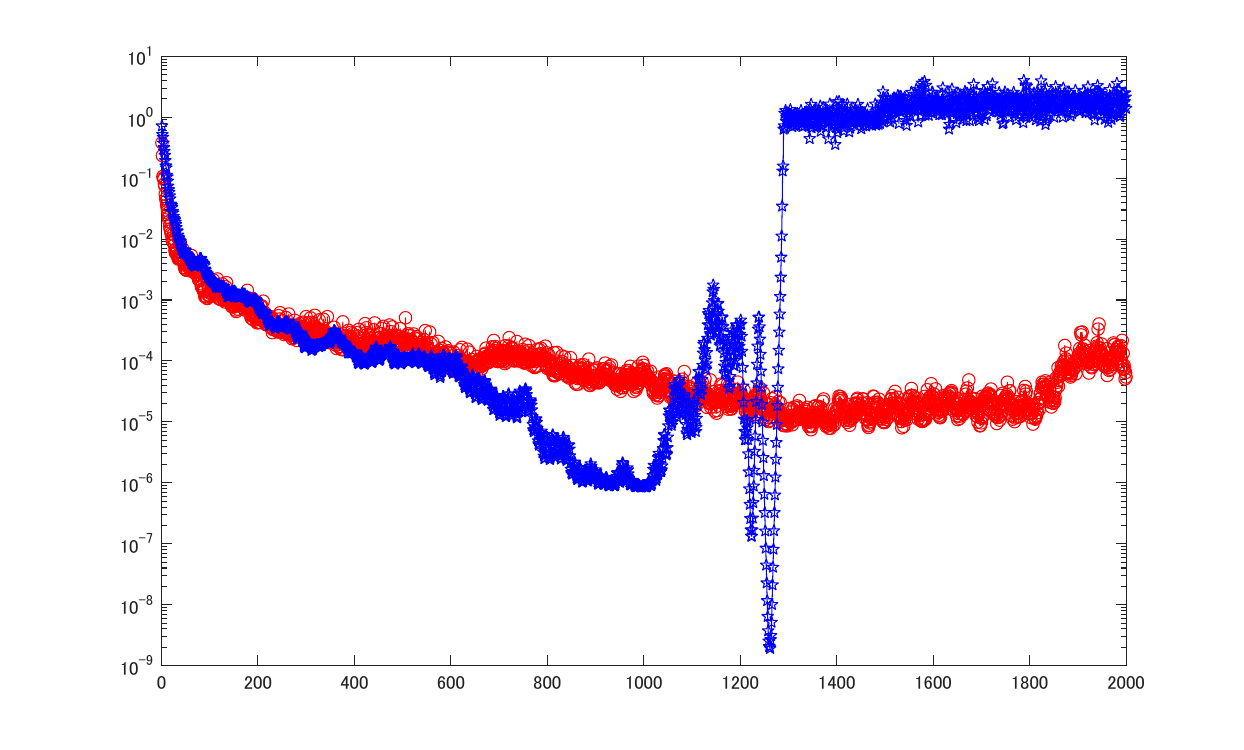}	
\caption{$\displaystyle \frac{\|A^{{\rm T}}\vector{r}_{k}\|_{2}}{\|A^{{\rm T}}\vector{b}\|_{2}}$ versus 
	the number of iterations for AB-RRGMRES using $B = A^{\rm T}$ (blue, $\star$)
and MINRES-QLP applied to the underdetermined least squares problem 
${\rm Maragal\_5T}$(red, $\circ$)}
\label{qlp_m5}
\end{center}
\end{figure}

\begin{figure}[htbp]
\begin{center}
\includegraphics[scale=0.55]{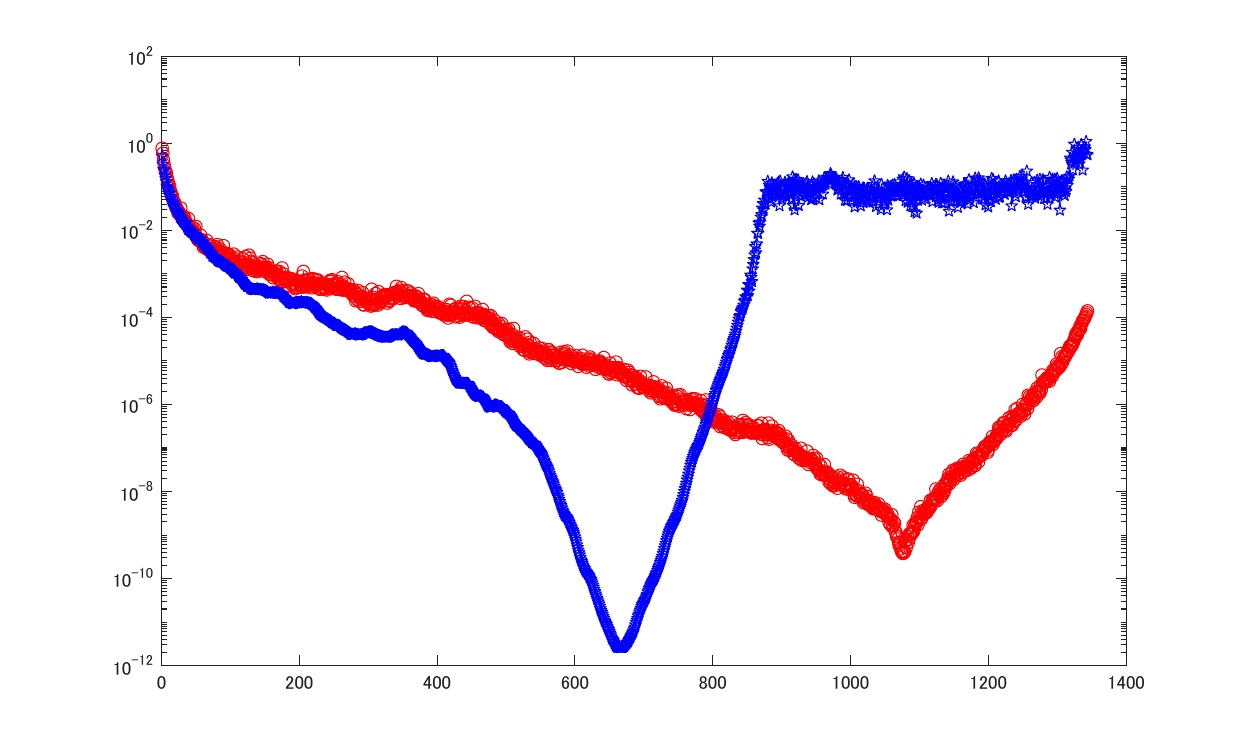}	
\caption{$\displaystyle \frac{\|A^{{\rm T}}\vector{r}_{k}\|_{2}}{\|A^{{\rm T}}\vector{b}\|_{2}}$ versus 
	the number of iterations for AB-RRGMRES using $B = A^{\rm T}$ (blue, $\star$)
and MINRES-QLP applied to the underdetermined least squares problem 
${\rm lp\_dfl001}$(red, $\circ$)}
\label{qlp_lp}
\end{center}
\end{figure}

We observe the following from Fig. \ref{qlp_m5} and \ref{qlp_lp}.

\begin{itemize}
\item For the least squares problem ${\rm Maragal\_5T}$,	
the minimum value of
$\displaystyle \frac{\|A^{{\rm T}}\vector{r}_{k}\|_{2}}{\|A^{{\rm T}}\vector{b}\|_{2}}$
for AB-RRGMRES with $B = A^{\rm T}$ is less than $10^{-8}$,
which is almost $10^{-4}$ times that of MINRES-QLP.
On the other hand, for the least squares problem ${\rm lp\_dfl001}$,
$\displaystyle \frac{\|A^{{\rm T}}\vector{r}_{k}\|_{2}}{\|A^{{\rm T}}\vector{b}\|_{2}}$
for AB-RRGMRES with $B = A^{\rm T}$ is almost $6.8 \times 10^{-3}$ times that of MINRES-QLP,
 and the former reaches the minimum quite earlier than the latter.

\end{itemize}

\section{Concluding remarks}\label{ConcL} 
We proposed applying the NR-SSOR right preconditioned RRGMRES 
for arbitrary singular systems and least squares problems,
and proved that the method
determines a solution of arbitrary singular systems
and least squares problems even if the systems are inconsistent.
Some numerical experiments on GP, index 2 inconsistent systems and 
underdetermined least squares problems showed that the proposed method
is more efficient and robust than
using $B = A^{\rm T}$, $B=\{{\rm diag}(A^{\rm T}A)\}^{-1} A$,
and obtains smaller residuals compared to MINRES-QLP
and the NR-SSOR right preconditioned GMRES.

\section*{Declarations}

Funding: We, Kota Sugihara and Ken Hayami, declare that there is no funding for the 
manuscript ``NR-SSOR right preconditioned RRGMRES for non-range-symmetric systems and rank-deficient 
rectangular least squares problems'' sent for possible publication to Numerical Algorithms Journal. 
More exactly: ``The research in this paper
received no specific grant from any funding agency in the public, private, or not-for-profit sectors."

\begin{appendices}
\section{Convergence theory of RRGMRES for nonsingular systems}\label{rrnonsg}
In this appendix, we show the convergence theory of RRGMRES for nonsingular systems,
which we think is new.

Lemma \ref{RReq} holds for both cases when $A$ is nonsingular and singular.
\begin{lemma}\label{RReq}
Assume $A \in \rnn$.
Let $\vector{x}_0$ be the initial iterate and $\vector{r}_0 = \vector{b} - A\vector{x}_0$.

When RRGMRES is applied to $A\vector{x} = \vector{b}$ with the initial iterate $\vector{x}_0$,
\begin{eqnarray*}
\|A\vector{x}_k - \vector{b}\|_{2}^{2}& =& \min_{\vector{y}\in \rk}
\|H_{k+1,k}\vector{y} - V_{k+1}^{\rm T}\vector{r}_0 \|_{2}^{2} \\
 & & + \|(I_{n} - V_{k+1}V_{k+1}^{\rm T})\vector{r}_0 \|_2^{2}.
\end{eqnarray*}
where $\vector{x}_{k} = \vector{x}_0 + V_{k}\vector{y}_{k}$ and $\vector{y}_{k}$ is the solution of the above minimization problem.
\end{lemma}
\begin{proof}
\begin{eqnarray*}
\|A\vector{x}_k - \vector{b}\|_{2}^{2}
& =& \min_{\vector{y}\in \rk}\|V_{k+1}H_{k+1,k}\vector{y} - V_{k+1}V_{k+1}^{\rm T}\vector{r}_0 \\
                                  & & + (V_{k+1}V_{k+1}^{\rm T} -I_{n})\vector{r}_0 \|_2^{2}.
\end{eqnarray*}
Here, 
since $V_{k+1}V_{k+1}^{\rm T}\vector{r}_0$ is the projection of $\vector{r}_0$ onto $\ran (V_{k+1})$, 
$(I_{n} - V_{k+1}V_{k+1}^{\rm T})\vector{r}_0$ is the projection of $\vector{r}_0$ onto $\ran (V_{k+1})^{\perp}$.\\
Therefore, $V_{k+1}(H_{k+1,k}\vector{y}_{k} - V_{k+1}^{\rm T}\vector{r}_0) \perp 
(I_{n} - V_{k+1}V_{k+1}^{\rm T})\vector{r}_0$.

Then, 
\begin{eqnarray*}
\|V_{k+1}H_{k+1,k}\vector{y} - V_{k+1}V_{k+1}^{\rm T}\vector{r}_0 
+ (V_{k+1}V_{k+1}^{\rm T} -I_{n})\vector{r}_0 \|_2^{2} 
                                      & = & \|V_{k+1}(H_{k+1,k}\vector{y}_{k} - V_{k+1}^{\rm T}\vector{r}_0)\|_{2}^{2}\\
                                       & & + \|(V_{k+1}V_{k+1}^{\rm T} - I_{n})\vector{r}_0\|_{2}^{2}
\end{eqnarray*}
Therefore, 
\begin{eqnarray*}
\|A\vector{x}_k - \vector{b}\|_{2}^{2}& =& \min_{\vector{y}\in \rk}\|V_{k+1}(H_{k+1,k}\vector{y}_{k} 
                                         - V_{k+1}^{\rm T}\vector{r}_0)\|_{2}^{2} + 
                                         \|(V_{k+1}V_{k+1}^{\rm T} - I_{n})\vector{r}_0 \|_{2}^{2} \\
                                      & = & \min_{\vector{y}\in \rk}\|H_{k+1,k}\vector{y}_{k} 
                                         - V_{k+1}^{\rm T}\vector{r}_0 \|_{2}^{2} + 
                                         \|(V_{k+1}V_{k+1}^{\rm T} - I_{n})\vector{r}_0 \|_{2}^{2} \\
\end{eqnarray*}
\end{proof}

\begin{theorem}\label{RRbr}
Let $A \in \rnn$ be a nonsingular matrix. 
Then, 
the iterate $\vector{x}_k$ of RRGMRES is an exact solution of $A\vector{x}=\vector{b}$ 
if 
$h_{k+1,k} = 0$.

Furthermore, RRGMRES converges in at most $n$ iterations.

\end{theorem}

\begin{proof}
Let $\vector{x}_0$ be the initial iterate and $\vector{r}_0 = \vector{b} - A\vector{x}_0$.

\begin{eqnarray*}
\|A\vector{x}_k - \vector{b}\|_{2}^{2}
& =& \min_{\vector{y}\in \rk}\|H_{k+1,k}\vector{y} - V_{k+1}^{\rm T}\vector{r}_0 \|_{2}^{2} \\
                                  & & + \|(I_{n} - V_{k+1}V_{k+1}^{\rm T})\vector{r}_0 \|_2^{2}.
\end{eqnarray*}

Assume that $h_{k+1,k}=0$.
Since $AV_{k} = V_{k+1}H_{k+1,k}$, we have
$AV_{k} = V_{k}H_{k,k}$.

Then, 
\begin{eqnarray*}
\|A\vector{x}_k - \vector{b}\|_{2}^{2}
& =& \min_{\vector{y}\in \rk}\|H_{k,k}\vector{y} - V_{k}^{\rm T}\vector{r}_0 \|_{2}^{2} \\
                                  &  &+ \|(I_{n} - V_{k}V_{k}^{\rm T})\vector{r}_0 \|_2^{2}.
\end{eqnarray*}

Firstly, we will prove that $H_{k,k}$ is nonsingular and \\
$\displaystyle \min_{\vector{y}\in \rk}\|H_{k,k}\vector{y} - V_{k}^{\rm T}\vector{r}_0 \|_{2} = 0$.

Here, since $h_{k+1,k} = 0$, we have $AV_{k} = V_{k}H_{k,k}$.
$A$ is nonsingular and $\rank V_{k} = k$. Then, $\rank AV_{k} = k$.

If $\rank H_{k,k} < k$, we have $\rank V_{k}H_{k,k} < k$. 
Therefore, $\rank H_{k,k} = k$. Since \\ $H_{k,k} \in \rkk$, $H_{k,k}$ is nonsingular.

Then, for $\displaystyle \vector{y}_{k}
 = \argmin_{\vector{y}\in \rk}\|H_{k,k}\vector{y} - V_{k}^{\rm T}\vector{r}_0 \|_{2}$,
$\|H_{k,k}\vector{y}_{k} - V_{k}^{\rm T}\vector{r}_0 \|_{2} = 0$.

Next, we will prove that
$\|(I_{n} - V_{k}V_{k}^{\rm T})\vector{r}_0 \|_2 = \vector{0}$ if $h_{k+1,k} = 0$.

Assume $(I_{n} - V_{k}V_{k}^{\rm T})\vector{r}_0 \neq \vector{0}$.
Then, $\vector{r}_0$ and all column vectors of $V_{k}$ are linearly independent.

Since $A$ is nonsingular, $A\vector{r}_0$ and all column vectors of $AV_{k}$ are linearly independent.
However, since $h_{k+1,k} = 0$, $A\vector{r}_0$ and all column vectors of $AV_{k}$ are linearly dependent.
This is a contradiction.
Therefore, if $h_{k+1,k} = 0$, we have $(I_{k} - V_{k}V_{k}^{\rm T})\vector{r}_0 = \vector{0}$

Here, there exists $k \leq n$ such $h_{k+1,k} = 0$.
Then, if $A$ is nonsingular, RRGMRES converges to a solution of $A\vector{x} = \vector{b}$ at the $k$th step.
Furthermore, since there exists $k \leq n$ such $h_{k+1,k} = 0$, RRGMRES converges in at most $n$ iterations.
\end{proof}

\section{Convergence theory of RRGMRES for singular systems}\label{rrsg}

Next, we show the convergence of RRGMRES for singular systems.
The statement except the number of iterations required for the convergence of RRGMRES in Theorem \ref{RRsg}
is shown in \cite[Theorem 3.2]{CLR}.
In this paper, we prove Theorem \ref{RRsg} in a different way from \cite[Theorem 3.2]{CLR}.

\begin{theorem}\label{RRsg}
RRGMRES determines 
a least squares solution 
of $$\min_{{\vector{x}}\in \mathbb{R}^{n}} \|\vector{b} - A\vector{x} \|_{2}$$
without breakdown for all $\vector{b}\in \mathbb{R}^{n}$
where $A$ may be singular if and only if
$\ran(A) = \ran(A^{\rm {T}})$. 

If $\vector{x}_0 \in \ran(A)$, the above least squares solution 
is a minimum norm solution.

Furthermore,
if $r = {\rm rank}A = {\rm dim}\ran(A) > 0$, 
RRGMRES converges in 
at most $r$ iterations.

\end{theorem}

\begin{proof}
Let 
\begin{eqnarray}
\vector{q}_{1},.,\vector{q}_{r} & : & \rm{orthonormal~basis}~\rm{of}~R(A) \label{eq:basis1} \\
\vector{q}_{r+1},.,\vector{q}_{n} & : & \rm{orthonormal~basis}~\rm{of}~R(A)^{{{\perp}}} \label{eq:basis2} 
\end{eqnarray} 

We define the matrices $Q_{1},Q_{2},Q$ as follows.
\begin{eqnarray*}
Q_{1} : & = &  [\vector{q}_{1},....,\vector{q}_{r}] \in \mathbb{R}^{n \times r}\\
Q_{2} : & = & [\vector{q}_{r+1},....,\vector{q}_{n}] \in \mathbb{R}^{n \times (n-r)} \\
Q     :  & = & [Q_{1}, Q_{2}] \in \mathbb{R}^{n \times n}
\end{eqnarray*}

Let
${\rm I}_{n} \in \mathbb{R}^{n \times n}$ be the identity matrix.
\begin{equation}\label{eq:orgth1}
Q^{\rm T}Q = {\rm I}_{n}
\end{equation}
holds from (\ref{eq:basis1}),(\ref{eq:basis2}).

Then, from (\ref{eq:orgth1}),
\begin{equation}\label{eq:orgth2}
QQ^{\rm T} = {\rm I}_{n}
\end{equation}
holds.

Let $\hat{A}  : = Q^{\rm T}AQ \in \mathbb{R}^{n \times n}$.
Since ${Q_{2}}^{\rm T}AQ = 0$ holds, we have
\begin{eqnarray*}
\hat{A} = \left[
\begin{array}{cc}
{Q_{1}}^{\rm T}AQ_{1} & {Q_{1}}^{\rm T}AQ_{2} \\
0                                 &  0                            
\end{array}
\right].
\end{eqnarray*}

We let $A_{11} = Q_1^{\rm T}A Q_1$ and $A_{12} = Q_1^{\rm T}A Q_2$.
Since $\ran(A) = \ran(A^{\rm T})$, $A_{12} = 0$ (\cite[Theorem 2.5]{HS}).
Furthermore, $A_{11}$ is nonsingular (\cite[Lemma 2.4]{HS}).

Thus, we can decompose RRGMRES
into the $\ran(A)$ component and the $\ran(A)^{\perp}$
component, as follows.\\*

{\bf Algorithm 4 : Decomposed RRGMRES method (general case)}
\vspace{12pt}
\\~~~~~$\ran(A)$ component \hspace{4cm} $\ran(A)^{\perp}$ component
\vspace{12pt}
\\1: $\vector{b}^1 = Q_1^{\rm T}\vector{b}$ \hspace{7cm} $\vector{b}^2 = Q_2^{\rm T}\vector{b}$ \\
2: Choose initial approximate solution $\vector{x}_0$. \\
3: Compute $ \vector{x}_0^1 =  Q_1^{\rm T}\vector{x}_0$. \hspace{4cm} $ 
\vector{x}_0^2 =  Q_2^{\rm T}\vector{x}_0$\\
4:  $\hat{\vector{r}}_0^1 
= A_{11}(\vector{b}^1 - A_{11}\vector{x}_0^1 - A_{12} \vector{x}_0^2) 
+ A_{12} \vector{b}^2$
 \hspace{5cm} $\hat{\vector{r}}_0^2 = \vector{0} $   \\
5:  $\|\hat{\vector{r}}_0\|_2 = \|\hat{\vector{r}}_0^1\|_2 
= \| A_{11}(\vector{b}^1 - A_{11}\vector{x}_0^1 - A_{12} \vector{x}_0^2) 
+ A_{12} \vector{b}^2 \|_2$ \\
6: $\vector{v}_1^1 : = \hat{\vector{r}}_0^1/\|\hat{\vector{r}}_0\|_{2} = 
 \hat{\vector{r}}_0^1/\|\hat{\vector{r}}_0^1\|_{2}$ \hspace{4cm}
$\vector{v}_1^2 : = \hat{\vector{r}}_0^2/\|\hat{\vector{r}}_0\|_2 = \vector{0}$\\
7: for $j=1$ until convergence do \\
8: \hspace{0.4cm}$h_{i,j}=
(\vector{v}_i^1,A_{11}\vector{v}_j^1 + A_{12}\vector{v}_j^2)
=  (\vector{v}_i^1,A_{11}\vector{v}_j^1)$~~$(i=1,2,...,j)$\\
9: \hspace{0.4cm}$\hat{\vector{v}}_{j+1}^1 = A_{11}\vector{v}_j^1  
- \sum_{i=1}^j h_{i,j}\vector{v}_i^1$ 
\hspace{2cm}
$\hat{\vector{v}}_{j+1}^2 = -{\sum_{i=1}^j} h_{i,j}\vector{v}_i^2 = \vector{0}$ \\
10: \hspace{0.4cm}$h_{j+1,j}
= \|\vector{v}_{j+1}\|_2 
= (\|\hat{\vector{v}}_{j+1}^1\|_2^2 + \|\hat{\vector{v}}_{j+1}^2 \|_2^2)^{\frac{1}{2}}
= \|\hat{\vector{v}}_{j+1}^1\|_2$.
~~If $h_{j+1,j}= 0$, go to line 14. \\
11 : Form the approximate solution\\
$\vector{x}_j^1 = \vector{x}_0^1 + [\vector{v}_1^1,....,\vector{v}_j^1]\vector{y}_j$
\hspace{2cm} $\vector{x}_j^2 = \vector{x}_0^2 + [\vector{v}_1^2,....,\vector{v}_j^2]\vector{y}_j
=\vector{x}_0^2 $ \\
where $\vector{y} = \vector{y}_{j+1}$ minimizes $\|\vector{r}_j^1\|_{2} $
$= \|\vector{b}^1 - A_{11}\vector{x}_0^1 - A_{11}[\vector{v}_{1}^1,....,\vector{v}_j^1]\vector{y}_j \\
- A_{12}\vector{x}_0^2 - A_{12}[\vector{v}_{1}^2,....,\vector{v}_j^2]\vector{y}_j\|_{2}$\\
$= \|\vector{b}^1 - A_{11}\vector{x}_j^1 - A_{12}\vector{x}_0^2\|_{2}$. ~~If $\|\vector{r}_j^1 \|_{2}$ = 0, go to line 14.\\
12: $\vector{v}_{j+1}^1 = \hat{\vector{v}}_{j+1}^1/h_{j+1,j}$ \hspace{4cm} $\vector{v}_{j+1}^2 
= \hat{\vector{v}}_{j+1}^2/h_{j+1,j} = \vector{0}$\\
13:end do \\
14: k :=j\\
15: Form the approximate solution\\
$\vector{x}_k^1 = \vector{x}_0^1 + [\vector{v}_1^1,....,\vector{v}_k^1]\vector{y}_k$
\hspace{2cm} $\vector{x}_k^2 = \vector{x}_0^2 + [\vector{v}_1^2,....,\vector{v}_k^2]\vector{y}_k
= \vector{x}_0^2$ \\
where $\vector{y} = \vector{y}_{k+1}$ minimizes $\|\vector{r}_k^1\|_{2} $
$= \|\vector{b}^1 - A_{11}\vector{x}_0^1 - A_{11}[\vector{v}_{1}^1,....,\vector{v}_k^1]\vector{y}_k \\
- A_{12}\vector{x}_0^2 - A_{12}[\vector{v}_{1}^2,....,\vector{v}_k^2]\vector{y}_k\|_{2}$\\
$= \|\vector{b}^1 - A_{11}\vector{x}_k^1 - A_{12}\vector{x}_0^2\|_{2}$. \\
16: $\vector{x}_k = Q_{1}\vector{x}_k^1 + Q_{2}\vector{x}_k^2 
= Q_{1}\vector{x}_k^1 + Q_{2}\vector{x}_0^2$ 

$\ran(\hat{A}) = \ran({\hat{A}}^{\rm {T}})$
is equivalent to $\ran(A) = \ran(A^{\rm {T}})$.

${Q_{1}}^{\rm T}AQ_{2}=0$ is equivalent to $\ran(\hat{A}) = \ran({\hat{A}}^{\rm {T}})$.
Hence, $A_{12}={Q_{1}}^{\rm T}AQ_{2}=0$ holds if and only if $\ran(A) = \ran(A^{\rm {T}})$.

If $A_{12} = 0$, the decomposed RRGMRES further simplifies as follows.
\vspace{12pt}

{\bf Algorithm 5: Decomposed RRGMRES method (Case ~$\ran(\hat{A}) = \ran({\hat{A}}^{\rm {T}})$)}
\vspace{12pt}
\\~~~~~$\ran(A)$ component \hspace{4cm} $\ran(A)^{\perp}$ component
\vspace{12pt}
\\1: $\vector{b}^1 = Q_1^{\rm T}\vector{b}$ \hspace{7cm} $\vector{b}^2 = Q_2^{\rm T}\vector{b}$ \\
2: Choose initial approximate solution $\vector{x}_0$. \\
3: Compute $ \vector{x}_0^1 =  Q_1^{\rm T}\vector{x}_0$. \hspace{4cm} $ 
\vector{x}_0^2 =  Q_2^{\rm T}\vector{x}_0$\\
4:  $\hat{\vector{r}}_0^1 
= A_{11}(\vector{b}^1 - A_{11}\vector{x}_0^1)$
 \hspace{5cm} $\hat{\vector{r}}_0^2 = \vector{0} $   \\
5:  $\|\hat{\vector{r}}_0\|_2 = \|\hat{\vector{r}}_0^1\|_2 
= \| A_{11}(\vector{b}^1 - A_{11}\vector{x}_0^1)\|_2 $ \\
6: $\vector{v}_1^1 : = \hat{\vector{r}}_0^1/\|\hat{\vector{r}}_0\|_{2} = 
 \hat{\vector{r}}_0^1/\|\hat{\vector{r}}_0^1\|_{2}$ \hspace{4cm}
$\vector{v}_1^2 : = \hat{\vector{r}}_0^2/\|\hat{\vector{r}}_0\|_2 = \vector{0}$\\
7: for $j=1$ until convergence do \\
8: \hspace{0.4cm}$h_{i,j}=
  (\vector{v}_i^1,A_{11}\vector{v}_j^1)$~~$(i=1,2,...,j)$\\
9: \hspace{0.4cm}$\hat{\vector{v}}_{j+1}^1 = A_{11}\vector{v}_j^1  
- \sum_{i=1}^j h_{i,j}\vector{v}_i^1$ 
\hspace{2cm}
$\hat{\vector{v}}_{j+1}^2 = -{\sum_{i=1}^j} h_{i,j}\vector{v}_i^2 = \vector{0}$ \\
10: \hspace{0.4cm}$h_{j+1,j}
= \|\vector{v}_{j+1}\|_2 
= (\|\hat{\vector{v}}_{j+1}^1\|_2^2 + \|\hat{\vector{v}}_{j+1}^2 \|_2^2)^{\frac{1}{2}}
= \|\hat{\vector{v}}_{j+1}^1\|_2$.
~~If $h_{j+1,j}= 0$, go to line 14. \\
11 : Form the approximate solution\\
$\vector{x}_j^1 = \vector{x}_0^1 + [\vector{v}_1^1,....,\vector{v}_j^1]\vector{y}_j$
\hspace{2cm} $\vector{x}_j^2 = \vector{x}_0^2 + [\vector{v}_1^2,....,\vector{v}_j^2]\vector{y}_j
=\vector{x}_0^2 $ \\
where $\vector{y} = \vector{y}_{j+1}$ minimizes $\|\vector{r}_j^1\|_{2} $
$= \|\vector{b}^1 - A_{11}\vector{x}_0^1 - A_{11}[\vector{v}_{1}^1,....,\vector{v}_j^1]\vector{y}_j \|_{2}$\\
$= \|\vector{b}^1 - A_{11}\vector{x}_j^1 \|_{2}$. ~~If $\|\vector{r}_j^1 \|_{2}$ = 0, go to line 14.\\
12: $\vector{v}_{j+1}^1 = \hat{\vector{v}}_{j+1}^1/h_{j+1,j}$ \hspace{4cm} $\vector{v}_{j+1}^2 
= \hat{\vector{v}}_{j+1}^2/h_{j+1,j} = \vector{0}$\\
13:end do \\
14: k :=j\\
15: Form the approximate solution\\
$\vector{x}_k^1 = \vector{x}_0^1 + [\vector{v}_1^1,....,\vector{v}_k^1]\vector{y}_k$
\hspace{2cm} $\vector{x}_k^2 = \vector{x}_0^2 + [\vector{v}_1^2,....,\vector{v}_k^2]\vector{y}_k
= \vector{x}_0^2$ \\
where $\vector{y} = \vector{y}_{k+1}$ minimizes $\|\vector{r}_k^1\|_{2} $
$= \|\vector{b}^1 - A_{11}\vector{x}_0^1 - A_{11}[\vector{v}_{1}^1,....,\vector{v}_k^1]\vector{y}_k \|_{2}$\\
$= \|\vector{b}^1 - A_{11}\vector{x}_k^1\|_{2}$. \\
16: $\vector{x}_k = Q_{1}\vector{x}_k^1 + Q_{2}\vector{x}_k^2 
= Q_{1}\vector{x}_k^1 + Q_{2}\vector{x}_0^2$ 

\vspace{12pt}
The $\ran(A)$ component of RRGMRES is essentially equivalent to RRGMRES applied to $A_{11}\vector{x}^1=\vector{b}^1$.
Since  $\vector{x}_k = Q_{1}\vector{x}_k^1 + Q_{2}\vector{x}_k^2=Q_{1}\vector{x}_k^1 +  Q_{2}\vector{x}_0^2$,
RRGMRES is essentially equivalent to RRGMRES applied to $A_{11}\vector{x}^1=\vector{b}^1$.
If $\ran(\hat{A}) = \ran(\hat{A}^{\rm {T}})$, $A_{11}$ is nonsingular.
If $\ran(\hat{A}) = \ran(\hat{A}^{\rm {T}})$, RRGMRES gives a least squares solution for all $\vector{b}$.

Since RRGMRES is essentially equivalent to RRGMRES applied to the nonsingular system $A_{11}\vector{x}^1=\vector{b}^1$
and $A_{11} \in \rrr$, RRGMRES converges in at most $r$ iterations from Theorem \ref{RRbr}.

If $\vector{x}_0 \in \ran(A)$, 
$\vector{x} \in \ran(A)$ holds for the solution $\vector{x}$ of RRGMRES. Since $\ran(A) = \ran(A^{\rm T})$,
$\vector{x} \in \ran(A^{\rm T})$. Since $\ran(A^{\rm T}) = \nul(A)^{\perp}$, $\vector{x} \in \nul(A)^{\perp}$.
Then, the solution $\vector{x}$ is the minimum norm solution.

Next, the necessity follows since if we assume that $A_{12} \neq 0$,
then there exists a $\vector{b}$ such that the decomposed RRGMRES  
(general case) breaks down at $0$th step without giving a least squares
solution. The details are given below.

Assume that $A_{12} \neq 0$. Then, there exists $\vector{s}^1 \neq \vector{0}$
such that $A_{11}\vector{s}^1 + A_{12}\vector{s}^2 = \vector{0}$
where $\vector{s}^1 \in \rr$ and $\vector{s}^2 \in {\bf R}^{n-r}$.

This can be shown as follows.

If $A_{11}$ is singular, there exists $\vector{s}^1 \neq \vector{0}$ such that
$A_{11}\vector{s}^1 = \vector{0}$, so let $\vector{s}^2 = \vector{0}$.

If $A_{11}$ is nonsingular, consider the following.
Since $A_{12} \neq 0$, there exists $(A_{12})_{i,j} \neq 0$, so that
$A_{12}\vector{e}_j \neq \vector{0}$, where $\vector{e}_j \in {\bf R}^{n-r}$
is the $j$th unit vector. Let $\vector{s}^2 = \vector{e}_j \neq \vector{0}$,
so that $A_{12}\vector{s}^2 \neq \vector{0}$.
Then,
let $\vector{s}^1 = - A_{11}^{-1} A_{12}\vector{s}^2 \neq \vector{0}$.
Thus, we have $A_{11}\vector{s}^1 + A_{12}\vector{s}^2 = \vector{0}$,
where $\vector{s}^1 \neq \vector{0}$.

Thus, let 
$\vector{b}^1 = \vector{s}^1 + A_{11}\vector{x}_0^1 + A_{12} \vector{x}_0^2$
and $\vector{b}^2 = \vector{s}^2$ in the decomposed RRGMRES (general case).

Then,
in the $\ran(A)$ component of the line 4 of the algorithm 4 of the decomposed 
RRGMRES (general case),
\begin{eqnarray*}
\hat{\vector{r}}_0^1 
& = &  A_{11}(\vector{b}^1 - A_{11}\vector{x}_0^1 - A_{12} \vector{x}_0^2) 
+ A_{12} \vector{b}^2 \\
                      & = & A_{11}\vector{s}^1 + A_{12}\vector{s}^2 \\
                      & = & \vector{0}
\end{eqnarray*}
Then, the algorithm terminates. That is, $\vector{v}_1^1$, $\vector{v}_1^2$ are not generated.
However, 
$\vector{r}_0^1 = \vector{b}^1 - A_{11}\vector{x}_0^1 - A_{12}\vector{x}_0^2 = \vector{s}^1 \neq \vector{0}$.
Hence, $Q_1^{\rm T}\vector{r}_0 = \vector{r}_0^1 \neq \vector{0}$, so that
$A^{\rm T}\vector{r}_0 \neq \vector{0}$, which means that $\vector{x}_0$ is not a least squares solution.
Then, this is a contradiction.

\end{proof}

Let $\rank A = r$. 
We showed that
RRGMRES  was equivalent to RRGMRES for nonsingular system 
when $\ran(A) = \ran(A^{\rm T})$. 
If the coefficient matrix for this nonsingular system
is described as $A_{11}$,
$A_{11} \in \rrr$.
Furthermore, let $\vector{x}_0$ be the initial iterate and $\vector{r}_0 = \vector{b} - A\vector{x}_0$.

Here, ${\cal K^{\ast}}_k(A, A\vector{r}_0)$ is the $k$th Krylov subspace determined by $A$ and $A\vector{r}_0$, defined by
\[ {\cal K^{\ast}}_k(A, A\vector{r}_0) \equiv {\rm span} \{ A\vector{r}_0, \ldots, A^k \vector{r}_0 \}. \].

Similarly to the breakdown of GMRES \cite{BW},
we shall say RRGMRES $does$ $not$ $break$ $down$ at the $k$th step if dim$A({\cal K^{\ast}}_k) = k$.

For range-symmetric systems, RRGMRES does not break down due to rank deficiency of $A{\cal K^{\ast}}_k$ 
(\cite[Lemma 3.1]{CLR}).
Furthermore, for range-symmetric systems, the minimum singular value of the Hessenberg matrix of RRGMRES is 
larger than $\sigma_r(A)$ which is the smallest singular values of $A$ \cite{MR}.

Now, by using the theorem that RRGMRES for range-symmetric systems essentially is equivalent
to RRGMRES for nonsingular systems, 
we will show that the RRGMRES breaks down due to the degeneracy of ${\cal K^{\ast}}_k$, which occurs when 
dim${\cal K^{\ast}}_k <k$ and 
the minimum singular value of the Hessenberg of RRGMRES is 
larger than $\sigma_r(A)$ which is the smallest singular values of $A$. 
Furthermore, we will show that the approximate solution $\bx_{k}$ at the $k$th step of RRGMRES is 
exact (which means that RRGMRES converges) if $h_{k+1,k} = 0$, where $h_{k+1,k}$ is 
the $(k+1,k)$ element of the Hessenberg matrix
$H_{k+1,k}$ of RRGMRES.

\begin{theorem}\label{RRbr2}
For range-symmetric systems,
RRGMRES breaks down due to the degeneracy of ${\cal K^{\ast}}_k$, which occurs when 
dim${\cal K^{\ast}}_k <k$ and 
the smallest singular value of the Hessenberg matrix of RRGMRES is 
larger than $\sigma_r(A)$ which is the smallest singular values of $A$.
 
Furthermore, 
the approximate solution $\bx_{k}$ at the $k$th step of RRGMRES is 
exact (which means that RRGMRES converges) if 
$h_{k+1,k} = 0$, where $h_{k+1,k}$ is the $(k+1,k)$ element of the Hessenberg matrix.

\end{theorem}

\begin{proof}
Let $\rank A = r$.
	Assume $\ran(A) = \ran(A^{\rm T})$.

Let $\vector{q}_1, ..., \vector{q}_r$ be orthonormal bases of $\ran(A)$
and $\vector{q}_{r+1}, ..., \vector{q}_n$ be orthonormal bases of $\ran(A)^{\perp}$. 

We define the matrices $Q_{1},Q_{2},Q$ as follows.
\begin{eqnarray*}
Q_{1} : & = &  [\vector{q}_{1},....,\vector{q}_{r}] \in \mathbb{R}^{n \times r}\\
Q_{2} : & = & [\vector{q}_{r+1},....,\vector{q}_{n}] \in \mathbb{R}^{n \times (n-r)} \\
Q     :  & = & [Q_{1}, Q_{2}] \in \mathbb{R}^{n \times n}
\end{eqnarray*}
We define $A_{11} = Q_1^{\rm T}AQ_1$, $\bb^1 = Q_1^{\rm T}\bb$
 and $\bx^1 = Q_1^{\rm T}\bx$.
Furthermore, let $\vector{x}_0^1 = Q_1^{\rm T}\vector{x}_0$ and
$\vector{r}_0^1 = \vector{b}_0^1 - A_{11}\vector{x}_0^1$.
	Since $R(A) = R(A^{\rm T})$, $A_{11}$ is nonsingular and $A_{11} \in \rrr$.
	Furthermore, since $R(A) = R(A^{\rm T})$, RRGMRES applied to $A\bx = \bb$ 
is essentially equivalent to RRGMRES applied to the nonsingular range-symmetric system 
$A_{11}\bx^1 = \bb^1$.

Since $A_{11}$ is nonsingular, ${\rm dim} A_{11}{\cal K^{\ast}}_k = {\rm dim} {\cal K^{\ast}}_k$.
Therefore, RRGMRES applied to $A_{11}\bx^1 = \bb^1$
does not break down due to rank deficiency of $A_{11}{\cal K^{\ast}}_k(A_{11}, A_{11}\vector{r}_0^1)$.
That is , RRGMRES applied to $A_{11}\bx^1 = \bb^1$ breaks down due to the degeneracy of 
${\cal K^{\ast}}_k(A_{11}, A_{11}\vector{r}_0^1)$, which occurs when 
dim${\cal K^{\ast}}_k(A_{11}, A_{11}\vector{r}_0^1) <k$.
Since RRGMRES applied to $A\bx = \bb$ 
is essentially equivalent to RRGMRES applied to the nonsingular range-symmetric system 
$A_{11}\bx^1 = \bb^1$, 
RRGMRES breaks down due to the degeneracy of 
${\cal K^{\ast}}_k(A, A\vector{r}_0)$, which occurs when 
dim${\cal K^{\ast}}_k(A, A\vector{r}_0) <k$. 
 
We will prove that
for $A_{11} = Q_1^{\rm T}AQ_1$, the singular values $\sigma_i(A_{11}) (i= 1,..,r)$
of $A_{11}$ are the same as the singular values $\sigma_i(A) (i= 1, ..., r)$ 
of $A$.

We denote the singular value decomposition of A by
$\displaystyle A  =  U\Sigma V^{\rm T}$
where $U \in \rnn$ and $V \in \rnn$
are orthogonal matrices 
$U^{\rm T}U = UU^{\rm T} = V^{\rm T}V = VV^{\rm T} = I_n$, $I_n \in \rnn$ is the identity matrix,
$\Sigma = {\rm diag}(\sigma_1, \sigma_2, ...,\sigma_r, 0,0,...,0) \in \rnn$, and
$\sigma_i$ is the $i$th largest nonzero singular values of $A$.
Let $\Sigma_r = {\rm diag}(\sigma_1, \sigma_2, ..., \sigma_r)$.

Let $U = [U_1, U_2]$ and $V = [V_1, V_2]$, where the columns of $U_1 \in \rnr$ and
$U_2 \in \rnnr$ form orthonormal bases of $\ran(A) = \ran(U_1)$ and 
$\ran(A)^{\perp} = \nul(A^{\rm T}) = \ran(U_2)$ respectively, and the columns of 
$V_1 \in \rnr$ and $V_2 \in \rnnr$ form orthonormal bases of $\nul(A)^{\perp} = \ran(A^{\rm T}) = \ran(V_1)$
and $\nul(A) = \ran(V_2)$, respectively.

\begin{eqnarray*}
A & = & U \Sigma V^{\rm T} \\
  & = & 
\left[U_1, U_2 \right]
\left[
\begin{array}{cc}
\Sigma_r & 0 \\
0                                 &  0                            
\end{array}
\right]
\left[
\begin{array}{c} 
V_1^{\rm T} \\
V_2^{\rm T}
\end{array}
\right] 
\end{eqnarray*}

Here, $Q_1^{\rm T} U_2 = 0$ since the columns of $U_2 \in \rnnr$ form orthonormal bases of 
$\ran(A)^{\perp} = \ran(U_2)$ and the columns of $Q_1 \in \rnr$ form orthonormal bases of 
$\ran(A)$.
Furthermore, since the columns of $V_2$ form orthonormal bases of $\nul(A) = \ran(V_2)$, $\nul(A) = \nul(A^{\rm T})$ 
is equivalent to $\ran(A) = \ran(A^{\rm T})$ and $\nul(A^{\rm T}) = \ran(A)^{\perp}$,
$V_2^{\rm T} Q_1 = 0$.

Then,
\begin{eqnarray*}
A_{11} &= & Q_1^{\rm T}A Q_1 \\
       & = & 
\left[Q_1^{\rm T}U_1, Q_1^{\rm T}U_2 \right]
\left[
\begin{array}{cc}
\Sigma_r & 0 \\
0                                 &  0                            
\end{array}
\right]
\left[
\begin{array}{c} 
V_1^{\rm T}Q_1 \\
V_2^{\rm T}Q_1
\end{array}
\right] \\
& = & 
Q_1^{\rm T}U_1 \Sigma_r V_1^{\rm T}Q_1
\end{eqnarray*}

Then,
\begin{eqnarray*}
A_{11}A_{11}^{\rm T} & = & Q_1^{\rm T}U_1 \Sigma_r V_1^{\rm T}Q_1 Q_1^{\rm T}V_1 \Sigma_r U_1^{\rm T}Q_1.
\end{eqnarray*}
Here, $Q_1 Q_1^{\rm T} + Q_2 Q_2^{\rm T} = I_n$ where $I_n \in \rnn$ is the identity matrix.
Then,
\begin{eqnarray*}
V_1^{\rm T}Q_1 Q_1^{\rm T}V_1 & = & V_1^{\rm T}(I_n - Q_2 Q_2^{\rm T}) V_1 \\
                                & = & V_1^{\rm T}V_1 - V_1^{\rm T}Q_2 Q_2^{\rm T}V_1
\end{eqnarray*}
Since the columns of $V_1$ form orthonormal bases of $\nul(A)^{\perp} = \ran(A^{\rm T}) = \ran(V_1)$, 
$V_1^{\rm T}V_1 = I_r$ where $I_r \in \rrr$ is the identity matrix.
Furthermore, since the columns of $Q_2$ form orthonormal bases of $\ran(A)^{\perp}$, 
$V_1^{\rm T}Q_2 = 0$ and $Q_2^{\rm T}V_1 = 0$.

Then,
$V_1^{\rm T}Q_1 Q_1^{\rm T}V_1 = I_r$.

Since $[U_1, U_2] \in \rnn$ is an orthogonal matrix,
$U_1 U_1^{\rm T} + U_2 U_2^{\rm T} = I_n$.
Then, 
$Q_1^{\rm T}U_1 U_1^{\rm T}Q_1 + Q_1^{\rm T}U_2 U_2^{\rm T}Q_1 = Q_1^{\rm T}Q_1 = I_r$.
Since the columns of $U_2$ form orthonormal bases of $\ran(A)^{\perp}$ and 
the columns of $Q_1$ form orthonormal bases of $\ran(A)$, 
$Q_1^{\rm T}U_2 = 0$ and $U_2^{\rm T} Q_1 = 0$.

Then, 
$Q_1^{\rm T}U_1 U_1^{\rm T}Q_1 = I_r$.

Since $[Q_1, Q_2] \in \rnn$ is an orthogonal matrix,
$Q_1 Q_1^{\rm T} + Q_2 Q_2^{\rm T} = I_n$.
Then, 
$U_1^{\rm T}Q_1 Q_1^{\rm T}U_1 + U_1^{\rm T}Q_2 Q_2^{\rm T}U_1 = U_1^{\rm T}U_1 = I_r$.
Since the columns of $Q_2$ form orthonormal bases of $\ran(A)^{\perp}$ and 
the columns of $U_1$ form orthonormal bases of $\ran(A)$,
$U_1^{\rm T}Q_2 Q_2^{\rm T}U_1 = 0$.

Then,
$U_1^{\rm T}Q_1 Q_1^{\rm T}U_1 = I_r$.

Since $Q_1^{\rm T}U_1 U_1^{\rm T}Q_1 = I_r$ and $U_1^{\rm T}Q_1 Q_1^{\rm T}U_1 = I_r$,
$Q_1^{\rm T}U_1 \in \rrr$ and $U_1^{\rm T}Q_1 \in \rrr$ is an orthogonal matrix.

Then, the eigenvalues of $A_{11}A_{11}^{\rm T}$ are the diagonal element of $\Sigma_r^{2}$.
Thus, the singular values of $A_{11}$ are the same as the singular values $\sigma_i(A)(i=1,..,r)$ of $A$.

Since $A_{11}$ is nonsingular, \[\sigma_k(\tilde{H}_{k+1,k}) \geq \sigma_r(A_{11})\].
where $\tilde{H}_{k+1,k}$ is a Hessenberg matrix of RRGMRES applied to $A_{11}\bx^1 = \bb^1$.

Next, 
we will prove that $H_{k+1,k} = \tilde{H}_{k+1,k} (1 \leq k \leq r)$ 
where $H_{k+1,k}$, $\tilde{H}_{k+1,k}$ are Hessenberg matrices 
of RRGMRES applied to $A\bx = \bb$, $A_{11}\bx^1 = \bb^1$.

\begin{eqnarray*}
A  &= & 
\left[Q_1, Q_2 \right]
\left[
\begin{array}{cc}
A_{11} & A_{12} \\
0                                 &  0                            
\end{array}
\right]
\left[
\begin{array}{c} 
Q_1^{\rm T} \\
Q_2^{\rm T}
\end{array}
\right] \\
& = & 
Q_1 A_{11}Q_1^{\rm T} + Q_1 A_{12}Q_2^{\rm T}
\end{eqnarray*}

When $\ran(A) = \ran(A^{\rm T})$, $A_{12} = 0$.
Then, $A = Q_1 A_{11}Q_1^{\rm T}$ when $\ran(A) = \ran(A^{\rm T})$.

\begin{eqnarray*}
h_{i,j} & = & \vector{v}_{i}^{\rm T} A \vector{v}_j \\
        & = & (Q_1 \vector{v}_i^{1} + Q_2 \vector{v}_i^{2})^{\rm T} Q_1 A_{11}Q_1^{\rm T} (Q_1\vector{v}_j^{1}+Q_2\vector{v}_j^{2})\\
        & = & (Q_1 \vector{v}_i^{1} + Q_2 \vector{v}_i^{2})^{\rm T}Q_1 A_{11} I_r \vector{v}_j^{1} \\
        & = & (Q_1 \vector{v}_i^{1} + Q_2 \vector{v}_i^{2})^{\rm T}Q_1 A_{11} \vector{v}_j^{1} \\
        & = & (Q_1 \vector{v}_i^{1})^{\rm T}Q_1 A_{11} \vector{v}_j^{1} \\
        & = & (\vector{v}_i^{1})^{\rm T}Q_1^{\rm T}Q_1 A_{11} \vector{v}_j^{1} \\
        & = & (\vector{v}_i^{1})^{\rm T}A_{11} \vector{v}_j^{1}
\end{eqnarray*}

Then, 
$H_{k+1,k} = \tilde{H}_{k+1,k} (1 \leq k \leq r)$ where $H_{k+1,k}$, $\tilde{H}_{k+1,k}$ are Hessenberg matrices 
of RRGMRES applied to $A\bx = \bb$, $A_{11}\bx^1 = \bb^1$.

Furthermore,  $\sigma_r(A_{11})$ is equal to $\sigma_r(A)$.

Therefore, 
\begin{equation}\label{sgHbd}
\sigma_k(H_{k+1,k}) \geq \sigma_r(A)
\end{equation}
since \[\sigma_k(\tilde{H}_{k+1,k}) \geq \sigma_r(A_{11})\].

Here, when RRGMRES is applied to the nonsingular system $A_{11}\vector{x}^1=\vector{b}^1$, 
the iterate at $k$step is exact if $h_{k+1,k} = 0$. 

Furthermore, the values of the $(k+1, k)$ element of the Hessenberg matrix of 
RRGMRES applied to $A_{11}\bx^1 = \bb^1$ is equal to that of 
the $(k+1, k)$ element of the Hessenberg matrix of 
RRGMRES applied to $A\bx = \bb$.

Since RRGMRES applied to $A\bx = \bb$ is equivalent to
RRGMRES applied to the nonsingular system $A_{11}\vector{x}^1=\vector{b}^1$, 
the iterate at $k$step is exact if $h_{k+1,k} = 0$.  
\end{proof}

\end{appendices}

\end{document}